\documentclass[11pt,letterpaper,reqno]{article}
\usepackage[letterpaper, left = 1in, right = 1in, top = 1.25 in, bottom = 1.25in]{geometry}
\usepackage[affil-it]{authblk}
\usepackage[sc]{mathpazo}

\usepackage{graphicx}%
\usepackage{multirow}%
\usepackage{amsmath,amssymb,amsfonts}%
\usepackage{amsthm}%
\usepackage{mathrsfs}%
\usepackage[title]{appendix}%
\usepackage{xcolor}%
\usepackage{textcomp}%
\usepackage{manyfoot}%
\usepackage{booktabs}%
\usepackage{algorithm}%
\usepackage{algorithmicx}%
\usepackage{algpseudocode}%
\usepackage{listings}%
\usepackage{bbm}
\usepackage[sort,numbers]{natbib}

\usepackage[linktocpage=true]{hyperref}
\hypersetup{
 colorlinks,
 linkcolor=blue,          
 citecolor=red!50!black,       
 filecolor=black,   
 urlcolor=black, 
 pdftitle={},
 pdfauthor={},
 pdfcreator={},
 pdfsubject={},
 pdfkeywords={}
}

\theoremstyle{thmstyleone}%
\newtheorem{theorem}{Theorem}%  meant for continuous numbers
\theoremstyle{thmstyletwo}%
\newtheorem{remark}{Remark}%
\newtheorem{lemma}{Lemma}
\newtheorem{claim}{Claim}
\newtheorem{corollary}{Corollary}
\theoremstyle{thmstylethree}%
\newtheorem{assumption}{Assumption}
\numberwithin{equation}{section}
\newcommand\blfootnote[1]{%
  \begingroup
  \renewcommand\thefootnote{}\footnote{#1}%
  \addtocounter{footnote}{-1}%
  \endgroup
}

\author{
  Bihan Chatterjee$^1$,  
  Siva Theja Maguluri$^2$,
  Debankur Mukherjee$^3$
}

\title{Higher-Order Approximations of Sojourn Times in M/G/1 Queues via Stein’s Method}
\date{}
\begin{document}
\maketitle

%\title[Article Title]{Higher Order Stein's Method for M/G/1 Queue}
%\author[1]{\fnm{Bihan} \sur{Chatterjee}}\email{bchatterjee9@gatech.edu}

%\author[1]{\fnm{Debankur} \sur{Mukherjee}}\email{debankur.mukherjee@isye.gatech.edu}
%\equalcont{These authors contributed equally to this work.}

%\author[1]{\fnm{Siva} \sur{ Theja Maguluri}}\email{siva.theja@gatech.edu}
%\equalcont{These authors contributed equally to this work.}

%\affil[1]{\orgdiv{H. Milton Stewart School of Industrial \& Systems Engineering}, \orgname{ Georgia Institute of Technology}, \orgaddress{\street{}, \city{Atlanta}, \postcode{30332}, \state{}, \country{USA}}}

\begin{abstract}
 We study the stationary sojourn time distribution in an M/G/1 queue operating under heavy traffic. It is known that the sojourn time converges to an exponential distribution in the limit. 
  Our focus is on obtaining pre-asymptotic, higher-order approximations that go beyond the classical exponential limit.  Using Stein’s method, we develop an approach based on higher-order expansions of the generator of the underlying Markov process. The key technical step is to represent higher-order derivatives in terms of lower-order ones and control the resulting error via derivative bounds of the Stein equation. Under suitable moment-matching conditions on the service distribution, we show that the approximation error decays as a high-order power of the slack parameter $\varepsilon=1-\rho$. Error bounds are established in the Zolotarev metric, which further imply bounds on the Wasserstein distance as well as the moments. Our results demonstrate that the accuracy of the exponential approximation can be systematically improved by matching progressively more moments of the service distribution.

\blfootnote{$^1$Georgia Institute of Technology, \emph{Email:} \href{bchatterjee9@gatech.edu}{bchatterjee9@gatech.edu}}
\blfootnote{$^2$Georgia Institute of Technology, \emph{Email:}  \href{mailto:siva.theja@gatech.edu}{siva.theja@gatech.edu}}
\blfootnote{$^3$Georgia Institute of Technology, \emph{Email:} \href{mailto:debankur.mukherjee@isye.gatech.edu}{debankur.mukherjee@isye.gatech.edu}}

% \blfootnote{\emph{Keywords and phrases}. Stein's Method, Diffusion Approximation, M/G/1 queue}
\end{abstract}

%\keywords{Stein's Method, Diffusion Approximation, M/G/1 queue}
%%\pacs[JEL Classification]{D8, H51}
%%\pacs[MSC Classification]{35A01, 65L10, 65L12, 65L20, 65L70}

\maketitle
    
\section{Introduction}
\label{sec1}
In service systems such as web servers, data centers, or call centers, the total time a task spends in the system, from arrival to service completion, is a key measure of both user-perceived quality and system efficiency. This end-to-end duration, known as the \emph{sojourn time}, includes both the waiting time in queue and the service time, making it a comprehensive performance metric. For single-sever queues with Poisson arrivals and exponential service requirements ($M/M/1$), the stationary sojourn time follows an exponential distribution. While the $M/M/1$ queue provides valuable analytical insights due to its simplicity, its assumption of exponentially distributed service times is often too restrictive for real-world systems. 
The $M/G/1$ queue relaxes this assumption by allowing for a general service time distribution, providing a more flexible and realistic framework. However, in the $M/G/1$ setting, the stationary sojourn time does not admit a closed-form expression. Instead, it is only characterized through its Laplace–Stieltjes transform via the Pollaczek–Khinchine formula~\cite{Pollaczek30, Khinchin32, Daigle05}. In the so called \emph{heavy-traffic regime}, it is known that the scaled stationary waiting time approaches an exponential distribution. Consequently, the scaled stationary sojourn time, being the sum of the stationary waiting and service times, also converges to an exponential distribution. These, however, are asymptotic results that apply only as the system load approaches capacity, and without explicit pre-limit error bounds their applicability in practice remains limited.

Motivated by the above, the goal of this paper is to characterize the pre-asymptotic behavior of the steady-state sojourn time distribution. Specifically, for a broad class of service time distributions, we study the convergence rates in terms of the distance between the stationary sojourn time distribution and its exponential limit. For this purpose, we work with the Zolotarev distance~\eqref{eq:zol}, a natural extension of the well-studied Wasserstein distance that controls convergence in distribution together with convergence of moments. 
Our main result provides conditions under which the distance to the limit decays at an arbitrarily high polynomial rate, i.e., it is bounded by a quantity of the form $\varepsilon^k$ for any $k >0$, where $\varepsilon=1-\rho$ is the slack parameter and $\rho$ is the server load.  
In particular, we show that if the first $k+1$ moments of the service distribution match those of the exponential distribution, then the Zolotarev-$k$ distance between the scaled steady-state sojourn time $\varepsilon D$ and its exponential limit $Z\sim \mathrm{Exp}(\mu)$ is $O(\varepsilon^k)$. This moment-matching requirement ensures the low-order moment alignment needed for a finite Zolotarev-$k$ distance, as this metric is sensitive to mismatches in low-order moments (see Remark \ref{remark1}). Practically, the results are most relevant for modest orders 
$k$, where matching a small number of moments yields a substantial improvement over the classical exponential heavy-traffic approximation.

Much of our analysis builds on Stein’s method, a framework introduced by Charles Stein in 1972 \cite{cs}. This method provides a general approach in probability theory for bounding the distance between two distributions with respect to a chosen probability metric, and is best known for establishing rates of convergence. In broad terms, the method consists of three steps. First, one formulates an equation that characterizes the target distribution, known as the \textit{Stein equation}. Second, one derives suitable \textit{gradient bounds} for the solutions of this equation. Finally, combining these ingredients with coupling techniques yields explicit error bounds between the pre-limit distribution and the target distribution. Over the years, Stein’s method has been developed in many directions to handle a wide variety of target distributions; a comprehensive overview can be found in the survey~\cite{NR}. \\

\noindent
\textbf{Overview of related works.} 
The analysis of $M/G/1$ and, more generally, $G/G/1$ queues has a long history, beginning with the classical works of Pollaczek \cite{Pollaczek30} and Khinchine \cite{Khinchin32}. Over the last five decades, a substantial literature has developed around the simplifications that arise in the study of such systems in the presence of heavy traffic. The earliest heavy-traffic approximation was obtained by Kingman \cite{king}, who derived the limiting distribution of the steady-state waiting time in the $G/G/1$ queue.

While asymptotic results provide valuable insights, many modern applications require accurate prelimit characterizations. In this regard, Stein’s method has emerged as a powerful tool. The generator comparison framework was first used in \cite{Ba} to approximate stationary distributions of Markov processes. There it was shown that when the approximating distribution coincides with the stationary distribution of a Markov process, Stein’s method can be used to establish distance bounds, for example, by recovering the Poisson distribution as the stationary distribution of a birth–death process. These ideas were extended in \cite{Ba2} to the multivariate normal distribution. Stein’s method has since been applied to single-server queues, as in \cite{WG}. Our work builds on this line by providing higher-order approximations for the $M/G/1$ queue, going beyond the results in \cite{WG}. 
Braverman, Dai, and Fang~\cite{BDF} further advanced this direction by introducing higher-order Stein’s method in the context of queueing theory, presenting both theoretical and numerical results across several models.

An alternative approach to refined approximations has come through diffusion limits. Classical works such as \cite{BX,KT} exploit this connection, and diffusion approximations for queueing systems have been developed extensively in \cite{MB,HA,TO,PE,RE,RJ}. Following Gurvich \cite{GU}, Stein’s method has also proved effective in rigorously establishing diffusion approximations. In related work, Braverman and Dai \cite{BAD,DB} provided approximations between Erlang queueing models and their limiting stationary distributions. More recent contributions include \cite{MH,BrAn}, which analyze the Join-the-Shortest-Queue (JSQ) model under different regimes using Stein’s method. Finally, higher-order approximations under the Zolotarev metric have been investigated in the context of normal distributions in \cite{GA,MF}.

Recent works \cite{BS2024, BS2025} have extended the generator-comparison approach of Stein’s method to queueing systems with generally distributed primitives by working with a representation in which the state is augmented with residual inter-event times, leading to a piecewise-deterministic Markov process  representation. In this setting, Braverman and Scully \cite{BS2024} analyze steady-state diffusion approximation error bounds for several canonical queueing models using a generator-comparison Stein approach formulated via the basic adjoint relationship (BAR), rather than a classical infinitesimal generator. They expand the BAR jump terms using Taylor’s theorem and then apply a Palm inversion step to convert event-stationary averages into time-stationary averages, which allows them to identify the corresponding diffusion generator and express the approximation error.

\section{Model Description and Preliminaries}
We consider an $M/G/1$ queue, a first-come-first-served system with a single server and infinite buffer capacity. Arrivals follow a Poisson process with rate $\lambda$, and service times are independent and identically distributed. The two performance measures of interest are the waiting time and the sojourn time. The waiting time is the duration a task spends in the system before service begins, while the sojourn time is the total time spent in the system, including its own service. 
In the following, $S$, $W$, and $D$ will denote random variables representing the service time of a task, the steady-state waiting time, and the steady-state sojourn time, respectively. The load of the system is $\rho = \lambda \mathbb{E}[S]$, and we define the slack parameter $\varepsilon = 1 - \rho$. Our convergence results will be expressed in terms of this parameter.

To measure the distance between distributions, we use the Wasserstein and Zolotarev metrics. The Wasserstein-$k$ distance between two probability measures $\mu$ and $\nu$ on $\mathbb{R}$ is defined as
\begin{align*}
    d_{W,k}(\mu,\nu)=\inf_{\gamma \in \Gamma(\mu,\nu)} \left[ \int |x-y|^k d \gamma \right]^{\frac{1}{k}} ,
\end{align*}
where $\Gamma(\mu,\nu)$ is the set of all couplings $\gamma$ of $\mu$ and $\nu$ on $\mathbb{R} \times \mathbb{R}$, i.e., $\int_{\mathbb R} \gamma(x,y)dy=\mu(x)$ and $\int_{\mathbb R} \gamma(x,y)dx=\nu(y).$
For the special case $k=1$, the Kantorovich–Rubinstein duality theorem~\cite{Villani} states that
$d_{W,1}(\mu,\nu)=\sup_{h \in Lip(1)} \big|\int h d\mu-\int h d\nu \big|.$
This motivates the Zolotarev distance of order $k$, defined for two probability distributions $\mu$ and $\nu$ as 
\begin{equation}\label{eq:zol}
    d_{Zol,k}(\mu,\nu)=\sup_{\|h^{(k)}\| \leq 1}  \bigg|\int h d\mu-\int h d\nu \bigg|,
\end{equation}
where $h^{(k)}$ denotes the $k$th derivative of $h$.
For random variables $X$ and $Y$ with laws $\mu$ and $\nu$ respectively, we write $d_{W,k}(X,Y)$ and $d_{Zol,k}(X,Y)$ instead, and when the measures are clear from the context, we write
\begin{align}\label{zolh}
    d_{Zol,k}(X,Y)=\sup_{\|h^{(k)}\| \leq 1} |\mathbb E[h(X)]-\mathbb E[h(Y)]|.
\end{align}
The Zolotarev-$k$ distance can be viewed as a natural generalization of the dual form of the Wasserstein-$1$ distance. %, with the latter corresponding to $k=1$. 
It is well known~\cite[Theorem 6.9]{Villani} that weak convergence together with convergence of moments is equivalent to the convergence in the Wasserstein $k$-distance. Moreover, the Zolotarev distance of order $k$ induces the same topology, namely convergence in distribution together with convergence of moments up to order $k$ \cite{MF,BDS}.
The Wasserstein and Zolotarev distances are also related through the  inequalities given in the following lemma, whose proof is deferred to Appendix \ref{app:c}.
\begin{lemma}\label{lemmadistancebound}
   For probability laws $\mu,\sigma$ we have 
    \begin{align}\label{bound Wasserstein}
        d_{W,k}(\mu, \sigma) \leq 2(2^{(k-2)}k.d_{Zol,k}(\mu, \sigma))^{\frac{1}{k}}.
    \end{align}
Moreover, if the first $k-1$ moments match, then
\begin{align}
    d_{Zol,k}(\mu,\sigma) \leq L_k d_{W,k}(\mu, \sigma)\bigg(\int |x|^k \mu(dx) +\int |x|^k \sigma(dx)\bigg)^{(k-1)/k},
\end{align}
where $L_k$ is a constant depending on $k$.
\end{lemma}
A recent work \cite{BB25} obtains tighter constants in the above inequalities in the special case of $k=2$, while also generalizing to higher dimensions. 
%Similar inequalities between $d_{zol,2}$ and $d_{W,2}$ in higher dimensions have been recently established in \cite{BB25}.

\section{Main Results}\label{sec2}

% The distribution of interest for us is the stationary sojourn time distribution of $M/G/1$ queues and we want to find pre-asymptotic 'higher order' approximations to these distributions. Here `higher order' refers to how the approximation error depends on the $\varepsilon = (1- \text{system load})$.
% In other words, our goal is to approximate the stationary sojourn time distribution $D$ such that it's the distance to the actual distribution can be bounded by higher powers of $\varepsilon$. The  notion of distance considered is the Zolotarev distance defined to be 
% \begin{align*}
%     d_{Zol,k}(X,Y)=\sup_{\|h^{(k)}\| \leq 1} |\mathbb E[h(X)]-\mathbb E[h(Y)]|.
% \end{align*}
% for nonnegative random variables $X$ and $Y$. 

We now turn to the stationary sojourn time distribution $D$ of the $M/G/1$ queue. Our aim is to quantify how accurately $D$, when appropriately scaled, can be approximated by the exponential distribution in heavy traffic, with the approximation error bounded by progressively higher powers of the slack parameter $\varepsilon$ introduced above.

Most existing analyses focus on the stationary waiting time distribution $W$. However, $W$ has a positive probability $\varepsilon$ of being exactly zero in steady state \cite{Pollaczek30}, which makes it difficult to approximate by a continuous distribution with error smaller than order $\varepsilon$. To bypass this limitation, we instead study the sojourn time $D$, which admits a continuous stationary distribution and is therefore more amenable to refined approximations.
A similar approach is used to bridge the gap between discrete distribution and continuous limits in the literature. For instance, to obtain Edgeworth-type expansions for central limit theorem of lattice distributions \cite[Theorem 2 in Chapter XVI, Section 4]{Feller} , one smooths out the scaled, centered empirical average to get higher order Normal approximations.
%{\color{blue}This is similar to Feller's \cite{Feller} analysis of lattice distributions. In particular,  \cite[Theorem 1 in Chapter XVI, Section 4]{Feller} shows that Edgeworth-type expansions hold uniformly for non-lattice distributions. 
%The argument, however, does not extend to lattice distributions. To obtain an expansion there, Feller \cite[Theorem 2 in Chapter XVI, Section 4]{Feller}  replaces the distribution function $F_n$ by its continuous polygonal approximant which smooths out the jumps of the lattice distribution and allows continuous expansions to apply. }

To state our main result, we impose two regularity conditions on the service time distribution $S$:
\begin{assumption}\label{assump1}
     \text{The distribution of $S$ has a density $g$ with $\lim_{y \to \infty}g(y)=0$.}
\end{assumption}
\begin{assumption}\label{assump2}
    The integral $\int_0^\infty |y^k(g'(y)+\mu g(y))|dy$ is bounded by a constant $C_1$. 
\end{assumption}
\noindent
Assumptions~\ref{assump1} and~\ref{assump2} are technical regularity conditions that arise from the proof technique. In our analysis, controlling higher-order remainder terms requires applying Taylor expansions inside integrals involving the service-time density and then performing integration by parts. These steps  generate expressions in which the density and its derivatives are multiplied by powers of the service-time variables. Assumption~\ref{assump1} is a mild tail condition ensuring that the service-time density decays at infinity and holds for most standard continuous service distributions. Assumption~\ref{assump2} is used to guarantee that the error term coming from the Taylor expansion is integrable. This condition is satisfied by a broad class of distributions including gamma, Weibull, phase-type, and lognormal distributions. In contrast, the assumptions may fail for  service-time distributions that do not admit a density, for distributions whose densities are not sufficiently smooth, or for heavy-tailed distributions that do not possess the finite moments required for the chosen approximation order.
With these assumptions in place, we can describe our main theorem. Informally, it shows that if the first $k+1$ moments of the service time distribution match those of the exponential distribution, then the scaled stationary sojourn time distribution becomes exponentially close in Zolotarev distance, with error decaying at least as fast as $\varepsilon^k$.
\begin{theorem}\label{thm1}
Let $D$ denote the steady-state sojourn time of an $M/G/1$ queue and let $Z \sim \mathrm{Exp}(\mu)$. Suppose Assumptions~\ref{assump1} and~\ref{assump2} hold, and the service time distribution satisfies $\mathbb{E}[S^i] = \mathbb{E}[Z^i]$ for all integers $i=1,2,\ldots,k+1$. Then for any $k\geq 2$,
\begin{align*}   
    d_{Zol,k}(\varepsilon D,Z)\leq C_2\varepsilon^k,
\end{align*}
where
\begin{align*}
    C_2=\frac{1}{\mu k!}C_1 +2  \left( \sum_{j=1}^{2^{k-1}}b_{k,j}   \left[ \int_0^\infty s^{d_{k,j}} dF(s)\right]\right).
\end{align*}
Here $F$ is the distribution function of the service time, and the constants $b_{k,1},b_{k,2}, \ldots, b_{k,2^{k-1}}$ and $d_{k,1},d_{k,2}, \ldots, d_{k,2^{k-1}}$ are defined in the proof of Lemma \ref{lemma general k}. Note that we have $3 \leq d_{k,1},d_{k,2}, \ldots, d_{k,2^{k-1}} \leq k+2$.
\end{theorem}

\begin{remark}\label{remark1}
Note that for the Zolotarev-$k$ distance between two distributions to be finite the first $k-1$ moments must coincide. This is because we can consider the test function $h(x)=c x^j$, where $j$ is an integer such that $1 \leq j<k$. Since $h^{(k)}(x)=0$, we have
\begin{align*}
    d_{Zol,k}(X,Y) \geq \bigg|\int (cx^j) d\mu-\int (cx^j) d\nu \bigg| \geq |c|\bigg|\int x^j d\mu-\int x^j d\nu \bigg|.
\end{align*}
Therefore, as $c \to \infty$, we will have $d_{Zol,k}(X,Y) \to \infty$, unless the first $k-1$ moments are equal.
In our setting, this is ensured by requiring the service time distribution to match the exponential distribution up to its $(k+1)^{th}$ moment. To see why, note that by \cite[Theorem 5]{Takacs}, the moments of the stationary waiting time satisfy the recursion
\begin{align}\label{wtrecursion}
    \mathbb E[W^k] = \frac\lambda{1-\rho}\sum_{i=1}^k\binom ki \frac{\mathbb E[S^{i+1}]}{i+1}\mathbb E[W^{k-i}],
\end{align}
and hence the $k^{th}$ moment of the sojourn time can be written as
\begin{align}\label{strecursion}
    \mathbb E[D^k] = \sum_{i=0}^k \binom ki \mathbb E[W^i]\mathbb E [S^{k-i}].
\end{align}
Thus, $\mathbb{E}[D^k]$ depends only on the first $k+1$ moments of the service time distribution. By assumption, these moments match those of the exponential distribution, and therefore the first $k$ moments of the sojourn time in the $M/G/1$ system agree with those of the $M/M/1$ system. This alignment of moments is precisely what allows us to derive a $k^{th}$ order error bound.
\end{remark}

\begin{remark}
An immediate implication of Theorem~\ref{thm1} is that by matching more moments of the service distribution with the exponential, one can drive the error in the Zolotarev metric to decay at arbitrarily high polynomial rates in $\varepsilon$. In other words, the classical exponential approximation for heavy-traffic sojourn times can be systematically refined to achieve increasing levels of accuracy.
\end{remark}

\begin{remark}
A key technical contribution of this paper is the inductive framework developed to control the error terms that arise in applying Stein’s method under the Zolotarev metric. See Section~\ref{sec general k} and the discussion in Remark~\ref{rem:induction} for more details. To obtain bounds of order $\varepsilon^k$, one naturally encounters expansions involving derivatives of the Stein solution up to order $(k+2)$. Our approach systematically eliminates all intermediate derivatives by expressing each higher-order term inductively in terms of lower-order derivatives and the $(k+2)^{\text{nd}}$ derivative alone, following the ideas introduced in \cite{BDF}. This reduction is crucial, as it enables all error terms to be bounded using only uniform control on the $(k+1)^{\text{st}}$ and $(k+2)^{\text{nd}}$ derivatives of the Stein solution.
\end{remark}

Finally, using the relationship between Wasserstein and Zolotarev distances discussed earlier, we obtain the following corollary.
\begin{corollary}\label{cor2}
    If the conditions of Theorem~\ref{thm1} hold for some $k$, then the Wasserstein distance satisfies
    \begin{align*}
        d_{W,k}( \varepsilon D,Z) \leq 2(2^{(k-2)}k.C_2)^{\frac{1}{k}} \varepsilon.
    \end{align*}
\end{corollary}
\begin{proof}[Proof of Corollary \ref{cor2}]
    The claim follows directly from Theorem \ref{thm1} and %together with 
    inequality~\eqref{bound Wasserstein}.
\end{proof}

\begin{corollary}\label{cor3}
 If the conditions of Theorem~\ref{thm1} hold for some $k$, then 
 \begin{align*}
     \bigg|\mathbb E[(\varepsilon D)^k]-\frac{k!}{\mu^k}\bigg| \leq C_2 k!\varepsilon^k
 \end{align*}
 \end{corollary}
 \begin{proof}[Proof of Corollary \ref{cor3}]
     The proof follows from Theorem \ref{thm1} and \eqref{zolh} by taking $h(x)=\frac{x^k}{k!}$. 
 \end{proof}

We also provide an alternative approach in Section \ref{sec:altk=2} for the case $k=2$ of Theorem \ref{thm1}, based on the Lindley recursion representation. At present, this proof has been established only for $k=2$.

\section{Proofs}\label{sec:proofs}
This section is devoted to the proof of Theorem~\ref{thm1}. We begin by establishing two preliminary lemmas that provide the key technical ingredients for the argument. To build intuition, we first treat the case $k=2$ in Section~\ref{sec k=2}, which highlights the main ideas in a simpler setting. We then extend the reasoning to the general case $k \geq 2$ in Section~\ref{sec general k}.

Recall that our objective is to bound the Zolotarev distance 
\begin{align*}
    d_{Zol,k}(\varepsilon D,Z)= d_{Zol,k}(\varepsilon W+\varepsilon S,Z)=\sup_{\|h^{(k)}\| \leq 1}|\mathbb E[h(\varepsilon W+\varepsilon S)]-\mathbb E[(h(Z))]|.
\end{align*}
As outlined above, we approach this via Stein's method. For each test function $h$, let $f_h$ be the solution of the Stein equation for the exponential distribution 
\begin{align}\label{steq}
     G_{Z}f_h(x)= -\mu f_h'(x)+ f''_h(x)+\mu f'_h(0)=h(x)-\mathbb E[h(Z)],
\end{align}
for bounded smooth functions $f_h$ see, e.g., Dai~\cite{Dai17} for a tutorial overview and Varadhan~\cite[Chapter 16]{Varadhan11} for the RBM generator with Neumann boundary conditions. 
Consequently, \eqref{steq} yields 
\begin{align}\label{mn}
    |\mathbb E[h(\varepsilon D)]&-\mathbb E[h(Z)]|= |\mathbb E[h(\varepsilon W+\varepsilon S)]-\mathbb E[h(Z)]| \nonumber\\
    &=|\mathbb EG_{z}f_h(\varepsilon W+\varepsilon S)|=\Big|-\mu \mathbb E[f_h'(\varepsilon W+\varepsilon S)]+ \mathbb E[f_h''(\varepsilon W+\varepsilon S)]+\mu f_h'(0)\Big|.
\end{align}
Thus, bounding the Zolotarev distance reduces to bounding the right-hand side of \eqref{mn}. To achieve this, we employ the generator comparison framework introduced in \cite{BDF}. Since our result concerns sojourn time rather than waiting time, we work with a slightly modified stochastic process. 
Specifically, for a fixed $y \geq 0$, we consider the process $\varepsilon W_t+\varepsilon y$ where $W_t$ is the waiting time process. This corresponds to shifting the waiting time process by $y$ and then scaling it. The generator of this modified process is given in the following lemma, whose proof is deferred to Section \ref{sec:preliminarylemmaproof}.
%whose proof is deferred to Appendix~\ref{app:lem1}.
\begin{lemma}\label{lemma1}
The generator of the shifted waiting time process is given by
\begin{align}\label{gen}
    G_{\varepsilon W_t+\varepsilon y}f(x+\varepsilon y)&=\lambda \int_0^\infty (f(x+\varepsilon y +\varepsilon s)-f(x+\varepsilon y))dF(s)-\varepsilon f'(x+\varepsilon y)\mathbbm{1}_{\{x>0\}},
\end{align}
where $F$ is the distribution function of the service time $S$.
\end{lemma}
This generator will be the key tool in applying the comparison argument. To control the error terms that arise,%arising in this approach, 
we also require bounds on the derivatives of the Stein equation solution $f_h$ defined in \eqref{steq}. These are summarized in the following lemma. The proof is provided in Section \ref{sec:preliminarylemmaproof}.
\begin{lemma}\label{lemma3}
Let $Z$ be the exponential random variable with mean $1/\mu$ and let $f_h$ be the solution to~\eqref{steq}. Then for any $k\geq 2$,
\begin{align*}
       |f_h^{(k+1)}(x)| \leq \frac{ \|h^{(k)}\|}{\mu} \quad\text{and}\quad |f_h^{(k+2)}(x)| \leq 2\|h^{(k)}\|.
\end{align*}
\end{lemma}
We now return to the proof of Theorem~\ref{thm1}. The key step is to expand the generator using Taylor’s theorem to isolate the main terms in~\eqref{mn} and identify the residual error. A straightforward Taylor expansion, however, produces error terms of the form $\varepsilon^i$ for $3 \leq i \leq k$, each involving higher derivatives of the Stein solution $f_h$. 
These terms prevent us from establishing a $k^{th}$ order bound on the Zolotarev distance. To overcome this obstacle, we apply the generator recursively to carefully chosen test functions and exploit the moment-matching assumptions. This recursive structure cancels the troublesome error terms and ensures that the remaining terms can be controlled at the desired order.
 
As noted earlier, the proof of Theorem~\ref{thm1} is presented in two parts, corresponding to the cases $k=2$ and $k \geq 2$ (see Sections~\ref{sec k=2} and~\ref{sec general k}). The general proof involves several technical steps and is somewhat lengthy, so we separate out the case $k=2$ to improve readability. This base case offers a self-contained argument that already illustrates the key ideas in a simpler form. We then build on these insights to establish the result for arbitrary $k$ in Section~\ref{sec general k}.
\subsection{Proof for $k=2$ }\label{sec k=2}
A key ingredient in the proof will be the following analytical lemma, whose proof is presented at the end of this subsection.
\begin{lemma}\label{Lemma k=2}
 For any $f$ with bounded third derivative, there exist constants $b_{2,1},b_{2,2},d_{2,1},d_{2,2}$ with  $3 \leq d_{2,1},d_{2,2} \leq 4$ such that 
    \begin{align*}
        \bigg|\mathbb E[f''(\varepsilon W+\varepsilon y)]-\mu \mathbb E[f'(\varepsilon W+\varepsilon y)]-\varepsilon f''(\varepsilon y)
     +\mu f'(\varepsilon y)\bigg|\leq \varepsilon^2 \|f^{(4)}\| \bigg|\sum_{j=1}^{2}b_{2,j}  \left[ \int_0^\infty s^{d_{2,j}}dF(s)\right]\bigg|,
    \end{align*}
where $F$ denotes the distribution function of $S$.
\end{lemma}
It is worthwhile to note that only the derivative of order $4$ appear on the RHS and the dependence on $\varepsilon$ is purely of order $\varepsilon^2$. This will be essential in proving Theorem \ref{thm1}. To bound the Zolotarev $k-$distance, we only have control over the $k^{th}$ derivative of the test function $h$. By Lemma \ref{lemma3}, this translates to bounds on the $(k+1)^{th}$ and $(k+2)^{th}$ derivatives of the solution of the Stein equation~\eqref{steq}. Consequently, it is necessary that all error terms be expressed in terms of these derivatives of the Stein solution. We now proceed with the proof of Theorem \ref{thm1} for $k=2$.

\begin{proof}[Proof of Theorem \ref{thm1} for $k=2$]
Recall that our goal is to bound the quantity
\begin{align}\label{bd1}
    \bigg|-\mu \mathbb E[f_h'(\varepsilon W+\varepsilon S)]+ \mathbb E[f_h''(\varepsilon W+\varepsilon S)]+\mu f_h'(0)\bigg|.
\end{align}
We look at each of the terms separately starting with $ \mathbb E[f_h''(\varepsilon W+\varepsilon S)]$
\begin{align}\label{stg19.1}
\nonumber    \mathbb E[f_h''(\varepsilon W &+ \varepsilon S)]\\
\nonumber &=\int_0^\infty \mathbb E[f_h''(\varepsilon W+\varepsilon y)]g(y)dy\\ \nonumber
&=\int_0^\infty \left[\mathbb E[f_h''(\varepsilon W+\varepsilon y)]-\mu \mathbb E[f_h'(\varepsilon W+\varepsilon y)]-\varepsilon f_h''(\varepsilon y)
     +\mu f_h'(\varepsilon y)\ \right] g(y)dy\\ 
   +&\int_0^\infty\mu \mathbb E[f_h'(\varepsilon W+\varepsilon y)]g(y)dy+\int_0^\infty\varepsilon f_h''(\varepsilon y)g(y)dy-\int_0^\infty \mu f_h'(\varepsilon y)g(y)dy .
   % \nonumber \\
   % &+\int_0^\infty \left( \sum_{j=1}^{2^{k-1}}b_{k,j} \varepsilon^{k} \mathbb E \left[ \int_0^\infty s^{d_{k,j}}f^{(k+2)}(\eta ) dF(s)\right]\right) g(y)dy
\end{align}
The first integral will be bounded using Lemma~\ref{Lemma k=2}. Further note that the second integral is  $\mu \mathbb E[f_h'(\varepsilon W+\varepsilon S)]$ which also appears in \eqref{bd1}. Therefore we have to only evaluate the last two integrals. We start by integrating by parts
\begin{align*}
    \int_0^\infty\varepsilon f_h''(\varepsilon y)g(y)dy&=\varepsilon \left[ \frac{f_h'(\varepsilon y)}{\varepsilon}g(y)\bigg|_0^\infty-\int_0^\infty \frac{f_h'(\varepsilon y)}{\varepsilon }g'(y)dy\right]\\
    &=- f_h'(0)g(0)-\int_0^\infty f_h'(\varepsilon y)g'(y)dy.
\end{align*}
Hence the difference of the last two integrals in \eqref{stg19.1} is 
\begin{align}\label{stg20.1}
& \int_0^\infty\varepsilon f_h''(\varepsilon y)g(y)dy-\int_0^\infty \mu f_h'(\varepsilon y)g(y)dy \nonumber \\
&=- f_h'(0)g(0)-\left[\int_0^\infty f_h'(\varepsilon y)g'(y)dy+\int_0^\infty \mu f_h'(\varepsilon y)g(y)dy\right]\nonumber\\
&=-f_h'(0)g(0)-\left[\int_0^\infty f_h'(\varepsilon y)(g'(y)+\mu g(y))dy \right] \nonumber\\
&\stackrel{(a)}{=}-f_h'(0)g(0) 
-\left[\int_0^\infty \left[f_h'(0)+\varepsilon y f_h''(0)+\frac{\varepsilon^2y^2}{2}f_h^{(3)}(\xi)\right](g'(y)+\mu g(y))dy \right] \nonumber\\
&=-f_h'(0)g(0)-\int_0^\infty \frac{\varepsilon^2 y^2}{2}f_h^{(3)}(\xi)(g'(y)+\mu g(y))dy -\int_0^\infty f_h'(0)(g'(y)+\mu g(y))dy \nonumber\\
&-\int_0^\infty \varepsilon yf_h''(0)(g'(y)+\mu g(y))dy ,
\end{align}
where in $(a)$ we have Taylor expanded $f$ around 0. We look at each of the last two integrals in \eqref{stg20.1} separately. The first one,
\begin{align}\label{stg21.1}
    \int_0^\infty f_h'(0)(g'(y)+\mu g(y))dy&= f_h'(0) \int_0^{\infty}g'(y)+\mu f_h'(0) \int_0^{\infty}g(y)dy  
    =-f_h'(0)g(0)+\mu f_h'(0),
\end{align}
since $g$ is a pdf with non-negative support and  $\lim_{y \to \infty}g(y)=0$.
And the second one
\begin{align}\label{stg22.1}
    \int_0^\infty \varepsilon yf_h''(0)(g'(y)+\mu g(y))dy %\nonumber\\
    &\stackrel{(a)}{=}\varepsilon f_h''(0)\left[yg(y)\bigg|_0^\infty-\int_0^\infty  g(y)dy \right]+\varepsilon f_h''(0)\mu \int_0^\infty yg(y)dy \nonumber \\
    &\stackrel{(b)}{=}-\varepsilon f_h''(0)+\varepsilon f_h''(0)\mu \frac{1}{\mu} 
    =0,
\end{align}
where $(a)$ follows using integration by parts and in $(b)$ we have used the Assumption $\mathbb E[S]=\mathbb E[Z]=\frac{1}{\mu}$. Putting \eqref{stg21.1} and \eqref{stg22.1} in \eqref{stg20.1} finally gives the required terms:
\begin{align}\label{stg21.12}
   & \int_0^\infty\varepsilon f_h''(\varepsilon y)g(y)dy-\int_0^\infty \mu f_h'(\varepsilon y)g(y)dy \nonumber\\
&=-\mu f_h'(0) -\int_0^\infty \frac{\varepsilon^2 y^2}{2}f_h^{(3)}(\xi)(g'(y)+\mu g(y))dy .
\end{align}
Putting \eqref{stg21.12} in \eqref{stg19.1},
\begin{align*}
    \mathbb E[f_h''(\varepsilon W+\varepsilon S)]&=\mu \mathbb E[f_h'(\varepsilon W+\varepsilon S)]-\mu f_h'(0) -\int_0^\infty \frac{\varepsilon^2 y^2}{2}f_h^{(3)}(\xi)(g'(y)+\mu g(y))dy \\
  &+\int_0^\infty \left[E[f_h''(\varepsilon W+\varepsilon y)]-\mu E[f_h'(\varepsilon W+\varepsilon y)]-\varepsilon f_h''(\varepsilon y)
     +\mu f_h'(\varepsilon y)\ \right] g(y)dy.
\end{align*}
Therefore,
\begin{align*}
    &\bigg|-\mu \mathbb E[f_h'(\varepsilon W+\varepsilon S)]+ \mathbb E[f_h''(\varepsilon W+\varepsilon S)]+\mu f_h'(0)\bigg|\\
    &\leq \int_0^\infty \bigg|\frac{\varepsilon^2y^2}{2}f_h^{(3)}(\xi) (g'(y)+\mu g(y)) \bigg|dy\\
    &+\int_0^\infty\bigg|\mathbb E[f_h''(\varepsilon W+\varepsilon y)]-\mu E[f_h'(\varepsilon W+\varepsilon y)]-\varepsilon f_h''(\varepsilon y)
     +\mu f_h'(\varepsilon y)\bigg| g(y)dy\\
    &\stackrel{(a)} {\leq} \|f_h^{(3)}\|\frac{\varepsilon^2}{2}\int_0^\infty|y^2(g'(y)+\mu g(y))|dy+\|f_h^{(4)}\| \varepsilon^2\int_0^\infty \bigg|\left( \sum_{j=1}^{2}b_{2,j}  \left[ \int_0^\infty s^{d_{2,j}} dF(s)\right]\right) g(y) \bigg|dy\\
    &\leq C \varepsilon^2,
\end{align*}
where $C$ is a constant given by 
\begin{align*}
    C=\frac{1}{2\mu }C_1 +2  \left( \sum_{j=1}^{2}b_{2,j}   \left[ \int_0^\infty s^{d_{2,j}} dF(s)\right]\right) 
\end{align*}
with constants $b_{2,1},b_{2,2},d_{2,1},d_{2,2}$ with $3 \leq ,d_{2,1},d_{2,2} \leq 4$.  Here, $(a)$ follows from the Lemma \ref{Lemma k=2} and the last step follows from Lemma \ref{lemma3} as we have $\|h^{(2)}\| \leq 1 $ and therefore $\|f_h^{(3)}\| \leq \frac{1}{\mu}$ and $\|f_h^{(4)}\| \leq 2$. Therefore, from \eqref{mn} we get
\begin{align*}
    d_{Zol,2}(\varepsilon(W+S),Z)\leq C\varepsilon^2.
\end{align*}
\end{proof}

\begin{proof}[Proof of Lemma~\ref{Lemma k=2}]
 The proof involves the following key steps;
\begin{enumerate}
    \item First we Taylor expand the generator in Lemma \ref{lemma1} up to three terms.
    \item Next, we apply the generator to a carefully chosen function to write the third derivative in terms of the second and fourth derivatives. 
    This expression of the the third derivative then can be used in the Taylor expansion in Step 1  to get the induction hypothesis which involves only the first, second and fourth derivatives. We give the details next.
\end{enumerate}
From Lemma \ref{lemma1} we have 
\begin{align*}
     G_{\varepsilon W_t+\varepsilon y}f(x+\varepsilon y)
     &= \lambda \int_0^\infty (f(x+\varepsilon y +\varepsilon s)-f(x+\varepsilon y))dF(s)-\varepsilon f'(x+\varepsilon y)\mathbbm{1}_{\{x>0\}}\\
     &= \lambda \mathbb E[f(x+\varepsilon S+\varepsilon y)-f(x+\varepsilon y)]-\varepsilon f'(x+\varepsilon y)\mathbbm{1}_{\{x>0\}}
\end{align*}
\textbf{Step 1:} Taylor expanding $f$ around $x+\varepsilon y$ up to three terms,
\begin{align} \label{eqcomb1}
       &G_{\varepsilon W_t+\varepsilon y}f(x+\varepsilon y) \nonumber \\ 
       &=\lambda \varepsilon \mathbb E[S]f'(x+\varepsilon y)+\frac{1}{2}\lambda \varepsilon^2 \mathbb E[S^2]f''(x+\varepsilon y)\nonumber \\
     &+\frac{1}{6}\lambda \varepsilon^3 \mathbb E[S^3]f'''(x+\varepsilon y)-\varepsilon f'(x+\varepsilon y)+\varepsilon f'(x+\varepsilon y)\mathbbm{1}_{\{x=0\}}+\varepsilon f'(x+\varepsilon y)\mathbbm{1}_{\{x<0\}}\nonumber \\
     &+\frac{1}{24}\lambda \varepsilon^4 \int_0^\infty s^4 f^{(4)}(\eta)dF(s)\nonumber \\
     &\geq \lambda \varepsilon \mathbb E[S]f'(x+\varepsilon y)+\frac{1}{2}\lambda \varepsilon^2 \mathbb E[S^2]f''(x+\varepsilon y)\nonumber \\
     &+\frac{1}{6}\lambda \varepsilon^3 \mathbb E[S^3]f'''(x+\varepsilon y)-\varepsilon f'(x+\varepsilon y)+\varepsilon f'(x+\varepsilon y)\mathbbm{1}_{\{x=0\}}+\varepsilon f'(x+\varepsilon y)\mathbbm{1}_{\{x<0\}}\nonumber \\
     &-\frac{1}{24}\lambda \varepsilon^4  \|f^{(4)}\|\int_0^\infty s^4 dF(s).
\end{align}
Similarly,
\begin{align}\label{eqcomb2}
    G_{\varepsilon W_t+\varepsilon y}f(x+\varepsilon y)
      &\leq \lambda \varepsilon \mathbb E[S]f'(x+\varepsilon y)+\frac{1}{2}\lambda \varepsilon^2 \mathbb E[S^2]f''(x+\varepsilon y) \nonumber \\
     &+\frac{1}{6}\lambda \varepsilon^3 \mathbb E[S^3]f'''(x+\varepsilon y)-\varepsilon f'(x+\varepsilon y)+\varepsilon f'(x+\varepsilon y)\mathbbm{1}_{\{x=0\}}+\varepsilon f'(x+\varepsilon y)\mathbbm{1}_{\{x<0\}}\nonumber \\
     &+\frac{1}{24}\lambda \varepsilon^4  \|f^{(4)}\|\int_0^\infty s^4 dF(s).
\end{align}
For convenience we write \eqref{eqcomb1} and $\eqref{eqcomb2}$ together as 
\begin{align}\label{eqcomb2.1}
       G_{\varepsilon W_t+\varepsilon y}f(x+\varepsilon y) 
     &\stackrel{(\leq)}{\geq} \lambda \varepsilon \mathbb E[S]f'(x+\varepsilon y)+\frac{1}{2}\lambda \varepsilon^2 \mathbb E[S^2]f''(x+\varepsilon y)\nonumber \\
     &+\frac{1}{6}\lambda \varepsilon^3 \mathbb E[S^3]f'''(x+\varepsilon y)-\varepsilon f'(x+\varepsilon y)+\varepsilon f'(x+\varepsilon y)\mathbbm{1}_{\{x=0\}}+\varepsilon f'(x+\varepsilon y)\mathbbm{1}_{\{x<0\}}\nonumber \\
     &\stackrel{(+)}{-}\frac{1}{24}\lambda \varepsilon^4  \|f^{(4)}\|\int_0^\infty s^4 dF(s).
\end{align}
where we have used the notation $a \stackrel{(\leq)}{\geq} b \stackrel{(+)}{-} c$, to combine the two inequalities $a \leq b+c$ and $a \geq b-c$.

Next, we want to use $ \mathbb E[G_{\varepsilon W_t+\varepsilon y}f(\varepsilon W+\varepsilon y) ]=0$, for the steady state waiting time $W$. We first show that $ \mathbb E[|G_{\varepsilon W_t+\varepsilon y}f(\varepsilon W+\varepsilon y)| ] < \infty$. We have 
\begin{align*}
    \bigg|\int_0^\infty (f(u+\epsilon s)-f(u)) dF(s) \bigg| \leq \sum_{j=1}^2 \frac{\epsilon^j \mathbb E[S^j]}{j!}|f^{(j)}(u)|+\frac{\epsilon^3\mathbb E[S^3]}{3!}\|f^{(3)}\|.
\end{align*}
And since
\begin{align*}
    f''(u)=f''(0)+\int_{0}^{u}f^{(3)}(t)dt,
\end{align*}
we have 
\begin{align*}
    |f''(u)| \leq |f''(0)|+\int_{0}^{|u|}|f^{(3)}(t)|dt \leq |f''(0)| +\|f^{(3)}\||u| 
\end{align*}
and similarly
\begin{align}\label{genbound1}
    |f'(u)| \leq |f'(0)|+\int_0^{|u|} (|f''(0)| +\|f^{(3)}\|t)dt=|f'(0)|+|f''(0)||u|+\frac{\|f^{(3)}\|}{2}u^2.
\end{align}
That is,
\begin{align}\label{geenbound2}
    \mathbb E[|f''(\varepsilon W+\varepsilon y)|] \leq  |f''(0)| +\|f^{(3)}\| \epsilon (\mathbb E[W]+y)< \infty,
\end{align}
and 
\begin{align*}
    \mathbb E[|f'(\varepsilon W+\varepsilon y)|] \leq |f'(0)|+|f''(0)|\epsilon(\mathbb E[W]+y)+\frac{\|f^{(3)}\|}{2}\epsilon^2(\mathbb E[W^2]+2y \mathbb E[W]+y^2)< \infty,
\end{align*}
where the inequalities in \eqref{genbound1}-\eqref{geenbound2} follows from \eqref{wtrecursion} and bounded third derivative.
Hence we have 
\begin{align}\label{genbound3}
    \mathbb E[|G_{\varepsilon W_t+\varepsilon y}f(\varepsilon W+\varepsilon y)| ] \leq \lambda \bigg[ \sum_{j=1}^2 \frac{\epsilon^j \mathbb E[S^j]}{j!}\mathbb E[|f^{(j)}(\varepsilon W+\varepsilon y)|]+\frac{\epsilon^3\mathbb E[S^3]}{3!}\|f^{(3)}\|\bigg] +\epsilon \mathbb E[  |f'(\varepsilon W+\varepsilon y)|]< \infty.
\end{align}

Therefore for the steady state waiting time $W$ 
\begin{align*}
    \mathbb E[G_{\varepsilon W_t+\varepsilon y}f(\varepsilon W+\varepsilon y) ]=0,
\end{align*}
we have using $P(W=0)=\varepsilon$ and  $P(W<0)=0$ in \eqref{eqcomb2.1},
\begin{align}\label{stg23}
    0 &\stackrel{(\leq)} {\geq} \lambda \varepsilon \mathbb E[S]\mathbb E[f'(\varepsilon W+\varepsilon y)]+\frac{1}{2}\lambda \varepsilon^2 \mathbb E[S^2]\mathbb E[f''(\varepsilon W+\varepsilon y)] \nonumber\\
    &+\frac{1}{6}\lambda \varepsilon^3 \mathbb E[S^3]\mathbb E[f'''(\varepsilon W+\varepsilon y)]-\varepsilon \mathbb E[f'(\varepsilon W+\varepsilon y)]+\varepsilon^2 f'(\varepsilon y)\stackrel{(+)}{-}\frac{1}{24}\lambda \varepsilon^4 \|f^{(4)}\| \int_0^\infty s^4 dF(s) \nonumber\\
    &=\frac{\lambda}{\mu}\varepsilon \mathbb E[f'(\varepsilon W+\varepsilon y)]+\frac{1}{2}\lambda \varepsilon^2 \frac{2}{\mu^2}E[f''(\varepsilon W+\varepsilon y)] \nonumber\\
    &+\frac{1}{6}\lambda \varepsilon^3 \frac{6}{\mu^3}\mathbb E[f'''(\varepsilon W+\varepsilon y)]-\varepsilon E[f'(\varepsilon W+\varepsilon y)]+\varepsilon^2f'(\varepsilon y)\stackrel{(+)}{-}\frac{1}{24}\lambda \varepsilon^4 \|f^{(4)}\| \int_0^\infty s^4 dF(s) \nonumber\\
    &=\varepsilon E[f'(\varepsilon W+\varepsilon y)]\left(\frac{\lambda}{\mu}-1\right)+\frac{\lambda}{\mu}\frac{\varepsilon^2}{\mu}E[f''(\varepsilon W+\varepsilon y)] \nonumber\\
    &+\frac{\lambda \varepsilon^3}{\mu^3}\mathbb E[f'''(\varepsilon W+\varepsilon y)]+\varepsilon^2f'(\varepsilon y)\stackrel{(+)}{-}\frac{1}{24}\lambda \varepsilon^4 \|f^{(4)}\| \int_0^\infty s^4dF(s).
 \end{align}
\textbf{Step 2:} We now derive the third derivative in terms of the second and fourth derivative. % as given by the following claim.
\begin{claim}
    \begin{align*}
        \frac{\lambda \varepsilon^3}{\mu^3}\mathbb E[f'''(\varepsilon W+\varepsilon y)] \stackrel{(\leq)} {\geq} \frac{\varepsilon^3}{\mu}\mathbb E[f''(\varepsilon W+\varepsilon y)]-\frac{\varepsilon^3}{\mu}f''(\varepsilon y)\stackrel{(+)}{-}\frac{1}{6} \lambda \varepsilon^4 \frac{1}{\mu} \|f^{(4)}\|\int_0^\infty s^3(\eta )dF(s).
    \end{align*}
\end{claim}
\begin{proof}
Again From the generator in Lemma \ref{lemma1} we have
\begin{align}\label{e2.1}
    G_{\varepsilon W_t+\varepsilon y}f(x+\varepsilon y)&=\lambda \int_0^\infty (f(x+\varepsilon y +\varepsilon s)-f(x+\varepsilon y))dF(s)-\varepsilon f'(x+\varepsilon y)\mathbbm{1}_{\{x>0\}}\\
  \nonumber   &=\lambda \varepsilon \mathbb E[S]f'(x+\varepsilon y)+\frac{1}{2}\lambda \varepsilon^2 \mathbb E[S^2]f''(x+\varepsilon y)-\varepsilon f'(x+\varepsilon y)+\varepsilon f'(x+\varepsilon y)\mathbbm{1}_{\{x=0\}} \nonumber \\
  &+\varepsilon f'(x+\varepsilon y)\mathbbm{1}_{\{x<0\}}
  +\frac{1}{6}\lambda \varepsilon^3 \int_0^\infty s^3 f'''(\eta)dF(s),
\end{align}
% And since $\mathbb E[G_{\varepsilon W_t+\varepsilon y}f(\varepsilon W+\varepsilon y)]=0$, 
% \begin{align*}
%     0&=\lambda \varepsilon \mathbb E[S]\mathbb E[f'(\varepsilon W+\varepsilon y)]+\frac{1}{2}\lambda \varepsilon^2 \mathbb E[S^2]\mathbb E[f''(\varepsilon W+\varepsilon y)]-\varepsilon \mathbb E[f'(\varepsilon W+\varepsilon y)]+\varepsilon \mathbb E[f'(\varepsilon W+\varepsilon y)\mathbbm{1}(W=0)]\\
%      &+\frac{1}{6}\lambda \varepsilon^3 \mathbb E \left[\int_0^\infty s^3 f'''(\eta)dF(s)\right]\\
%      \stackrel{a}{=}&\varepsilon \mathbb E[f'(\varepsilon W+\varepsilon y)] \left( \frac{\lambda}{\mu}-1\right)+\frac{1}{2}\lambda \varepsilon^2 \frac{2}{\mu^2}\mathbb E[f''(\varepsilon W+\varepsilon y)]+\varepsilon^2 f'(\varepsilon y)+\frac{1}{6}\lambda \varepsilon^3 \mathbb E \left[\int_0^\infty s^3 f'''(\eta)dF(s)\right]\\
%      &=-\varepsilon^2 E[f'(\varepsilon W+\varepsilon y)]+\frac{\lambda}{\mu^2} \varepsilon^2\mathbb E[f''(\varepsilon W+\varepsilon y)]+\varepsilon^2 f'(\varepsilon y)+\frac{1}{6}\lambda \varepsilon^3 \mathbb E \left[\int_0^\infty s^3 f'''(\eta)dF(s)\right]\\
% \end{align*}
% where in ($a$) we have used $P(W=0)=1-\rho=\varepsilon$ and $\mathbb E[S]=\frac{1}{\mu}, \mathbb E[S^2]=\frac{2}{\mu^2}$
where we Taylor expanded $f$ around $x+\varepsilon y$. Therefore,
\begin{align*}
  & G_{\varepsilon W+\varepsilon y}\frac{\frac{1}{6}\lambda \varepsilon^3 \mathbb E[S^3]}{\frac{1}{2}\lambda \varepsilon^2 \mathbb E[S^2]}f'(x+\varepsilon y)\\
  &=\lambda \varepsilon \mathbb E[S].\frac{\frac{1}{6}\lambda \varepsilon^3 \mathbb E[S^3]}{\frac{1}{2}\lambda \varepsilon^2 \mathbb E[S^2]}f''(x+\varepsilon y)+\frac{1}{2}\lambda \varepsilon^2 \mathbb E[S^2].\frac{\frac{1}{6}\lambda \varepsilon^3 \mathbb E[S^3]}{\frac{1}{2}\lambda \varepsilon^2 \mathbb E[S^2]}f'''(x+\varepsilon y)\\
    &-\varepsilon.\frac{\frac{1}{6}\lambda \varepsilon^3 \mathbb E[S^3]}{\frac{1}{2}\lambda \varepsilon^2 \mathbb E[S^2]} f''(x+\varepsilon y)+\varepsilon \frac{\frac{1}{6}\lambda \varepsilon^3 \mathbb E[S^3]}{\frac{1}{2}\lambda \varepsilon^2 \mathbb E[S^2]}f''(x+\varepsilon y)\mathbbm{1}_{\{x=0\}}+\varepsilon \frac{\frac{1}{6}\lambda \varepsilon^3 \mathbb E[S^3]}{\frac{1}{2}\lambda \varepsilon^2 \mathbb E[S^2]}f''(x+\varepsilon y)\mathbbm{1}_{\{x<0\}}\\
    &+\frac{1}{6} \lambda \varepsilon^3 \frac{\frac{1}{6}\lambda \varepsilon^3 \mathbb E[S^3]}{\frac{1}{2}\lambda \varepsilon^2 \mathbb E[S^2]} \int_0^\infty s^3 f^{(4)}(\eta )dF(s)\\
    &\stackrel{(a)}{=}\varepsilon(\rho-1 ).\frac{\frac{1}{6}\lambda \varepsilon^3 \mathbb E[S^3]}{\frac{1}{2}\lambda \varepsilon^2 \mathbb E[S^2]}f''(x+\varepsilon y)+\frac{1}{6}\lambda \varepsilon^3 \mathbb E[S^3]f'''(x+\varepsilon y)+\varepsilon \frac{\frac{1}{6}\lambda \varepsilon^3 \mathbb E[S^3]}{\frac{1}{2}\lambda \varepsilon^2 \mathbb E[S^2]}  f''(x+\varepsilon y)\mathbbm{1}_{\{x=0\}}\\
    &+\varepsilon \frac{\frac{1}{6}\lambda \varepsilon^3 \mathbb E[S^3]}{\frac{1}{2}\lambda \varepsilon^2 \mathbb E[S^2]}  f''(x+\varepsilon y)\mathbbm{1}_{\{x<0\}}+\frac{1}{18} \lambda \varepsilon^4 \frac{\mathbb E[S^3]}{\mathbb E[S^2]}\int_0^\infty s^3 f^{(4)}(\eta )dF(s)\\
    &\stackrel{(b)}{=}-\varepsilon^2\frac{\frac{1}{6}\lambda \varepsilon^3 \mathbb E[S^3]}{\frac{1}{2}\lambda \varepsilon^2 \mathbb E[S^2]}f''(x+\varepsilon y)+\frac{1}{6}\lambda \varepsilon^3 \mathbb E[S^3]f'''(x+\varepsilon y)+\varepsilon \frac{\frac{1}{6}\lambda \varepsilon^3 \mathbb E[S^3]}{\frac{1}{2}\lambda \varepsilon^2 \mathbb E[S^2]} f''(x+\varepsilon y)\mathbbm{1}_{\{x=0\}}\\
    &+\varepsilon \frac{\frac{1}{6}\lambda \varepsilon^3 \mathbb E[S^3]}{\frac{1}{2}\lambda \varepsilon^2 \mathbb E[S^2]} f''(x+\varepsilon y)\mathbbm{1}_{\{x<0\}}+\frac{1}{18} \lambda \varepsilon^4 \frac{\mathbb E[S^3]}{\mathbb E[S^2]}\int_0^\infty s^3 f^{(4)}(\eta )dF(s)\\
   &\stackrel{(\geq)} {\leq} -\varepsilon^2\frac{\frac{1}{6}\lambda \varepsilon^3 \mathbb E[S^3]}{\frac{1}{2}\lambda \varepsilon^2 \mathbb E[S^2]}f''(x+\varepsilon y)+\frac{1}{6}\lambda \varepsilon^3 \mathbb E[S^3]f'''(x+\varepsilon y)+\varepsilon \frac{\frac{1}{6}\lambda \varepsilon^3 \mathbb E[S^3]}{\frac{1}{2}\lambda \varepsilon^2 \mathbb E[S^2]} f''(x+\varepsilon y)\mathbbm{1}_{\{x=0\}}\\
    &+\varepsilon \frac{\frac{1}{6}\lambda \varepsilon^3 \mathbb E[S^3]}{\frac{1}{2}\lambda \varepsilon^2 \mathbb E[S^2]} f''(x+\varepsilon y)\mathbbm{1}_{\{x<0\}}\stackrel{(-)}{+}\frac{1}{18} \lambda \varepsilon^4 \frac{\mathbb E[S^3]}{\mathbb E[S^2]} \|f^{(4)}\|\int_0^\infty s^3 dF(s),
\end{align*}
where we have used $\lambda \mathbb E[S]=\rho$ in $(a)$ and $1-\rho=\varepsilon$ in $(b)$. Since $E\left[| G_{\varepsilon W+\varepsilon y}f'(\varepsilon W+\varepsilon y)|\right]< \infty$, using the same arguments we used in \eqref{genbound3}, we have from the property of generators
\begin{align*}
   \mathbb E\left[ G_{\varepsilon W+\varepsilon y}\frac{\frac{1}{6}\lambda \varepsilon^3 \mathbb E[S^3]}{\frac{1}{2}\lambda \varepsilon^2 \mathbb E[S^2]}f'(\varepsilon W+\varepsilon y)\right]=0.
\end{align*}
Hence we have
\begin{align*}
    &\frac{1}{6}\lambda \varepsilon^3 \mathbb E[S^3]\mathbb E[f'''(\varepsilon W+\varepsilon y)]\\ &\stackrel{(\leq)} {\geq} \varepsilon^2\frac{\frac{1}{6}\lambda \varepsilon^3 \mathbb E[S^3]}{\frac{1}{2}\lambda \varepsilon^2 \mathbb E[S^2]}\mathbb E[f''(\varepsilon W+\varepsilon y)]-\varepsilon^2 \frac{\frac{1}{6}\lambda \varepsilon^3 \mathbb E[S^3]}{\frac{1}{2}\lambda \varepsilon^2 \mathbb E[S^2]} f''(\varepsilon y)       \stackrel{(+)}{-}\frac{1}{18} \lambda \varepsilon^4 \frac{\mathbb E[S^3]}{\mathbb E[S^2]} \|f^{(4)}\|\int_0^\infty s^3 dF(s)\\
        =  &\varepsilon^2\frac{\frac{1}{6}\lambda \varepsilon^3 \mathbb E[S^3]}{\frac{1}{2}\lambda \varepsilon^2 \mathbb E[S^2]}\mathbb E[f''(\varepsilon W+\varepsilon y)]-\varepsilon^3\frac{\frac{1}{6}\mathbb E[S^3]}{\frac{1}{2} \mathbb E[S^2]} f''(\varepsilon y)
        \stackrel{(+)}{-}\frac{1}{18} \lambda \varepsilon^4 \frac{\mathbb E[S^3]}{\mathbb E[S^2]} \|f^{(4)}\|\int_0^\infty s^3 dF(s),
\end{align*}
where in the first inequality we used $P(W=0)=\varepsilon$. Therefore,
\begin{align}\label{stg2}
    \frac{\lambda \varepsilon^3}{\mu^3}\mathbb E[f'''(\varepsilon W+\varepsilon y)] \stackrel{(\leq)} {\geq} \frac{\varepsilon^3}{\mu}\mathbb E[f''(\varepsilon W+\varepsilon y)]-\frac{\varepsilon^3}{\mu}f''(\varepsilon y)\stackrel{(+)}{-}\frac{1}{6} \lambda \varepsilon^4 \frac{1}{\mu} \|f^{(4)}\|\int_0^\infty s^3 dF(s).
\end{align}
\end{proof}
We now use the expression for the third derivative in \eqref{stg2} in \eqref{stg23} to get the result and the constants $b_{2,1}, b_{2,2}, d_{2,1},d_{2,2}$.
From \eqref{stg23} we have
\begin{align*}
    0
   &\stackrel{(\leq)} {\geq} \varepsilon E[f'(\varepsilon W+\varepsilon y)]\left(\frac{\lambda}{\mu}-1\right)+\frac{\lambda}{\mu}\frac{\varepsilon^2}{\mu}E[f''(\varepsilon W+\varepsilon y)]\\
    &+\frac{\lambda \varepsilon^3}{\mu^3}\mathbb E[f'''(\varepsilon W+\varepsilon y)]+\varepsilon^2f'(\varepsilon y)\stackrel{(+)}{-}\frac{1}{24}\lambda \varepsilon^4 \|f^{(4)}\| \int_0^\infty s^4 dF(s)\\
   & \mathrel{\substack{(a)\\ \leq \\ \geq}}-\varepsilon^2 E[f'(\varepsilon W+\varepsilon y)]+\frac{\rho}{\mu}\varepsilon^2E[f''(\varepsilon W+\varepsilon y)]+\frac{\varepsilon^3}{\mu}\mathbb E[f''(\varepsilon W+\varepsilon y)]-\frac{\varepsilon^3}{\mu}f''(\varepsilon y)+\varepsilon^2f'(\varepsilon y)\\
    &\stackrel{(+)}{-}\frac{1}{6} \lambda \varepsilon^4 \frac{1}{\mu}\|f^{(4)}\|\int_0^\infty s^3 dF(s)\stackrel{(+)}{-}\frac{1}{24}\lambda \varepsilon^4 \|f^{(4)}\|\int_0^\infty s^4dF(s)\\
    &=-E[f'(\varepsilon W+\varepsilon y)]+\frac{1}{\mu}(\rho+\varepsilon)E[f''(\varepsilon W+\varepsilon y)]-\frac{\varepsilon}{\mu}f''(\varepsilon y)+f'(\varepsilon y)\\
     &\stackrel{(+)}{-}\frac{1}{6} \lambda \varepsilon^2 \frac{1}{\mu}\|f^{(4)}\|\int_0^\infty s^3 dF(s)
     \stackrel{(+)}{-}\frac{1}{24}\lambda \varepsilon^2 \int_0^\infty s^4 f^{(4)}(\eta)dF(s),
\end{align*}
where in ($a$) we have used \eqref{stg2}. Therefore,
\begin{align}\label{stg3}
   \nonumber  &-\frac{1}{6} \lambda \varepsilon^2 \|f^{(4)}\|\int_0^\infty s^3 dF(s)-\frac{1}{24}\mu \lambda \varepsilon^2 \|f^{(4)}\| \int_0^\infty s^4 dF(s) \nonumber \\
   &\leq E[f''(\varepsilon W+\varepsilon y)]-\mu E[f'(\varepsilon W+\varepsilon y)]-\varepsilon f''(\varepsilon y)+\mu f'(\varepsilon y) \nonumber\\
     &\leq\frac{1}{6} \lambda \varepsilon^2 \|f^{(4)}\|\int_0^\infty s^3 dF(s)+\frac{1}{24}\mu \lambda \varepsilon^2 \|f^{(4)}\| \int_0^\infty s^4 dF(s).
\end{align}
Hence we have the lemma with $b_{2,1}=\frac{1}{6}\lambda, \ d_{2,1}=3,\ b_{2,2}=\frac{1}{24}\mu \lambda \text{ and }\  d_{2,2}=4$. 

\end{proof}
\subsection{Proof for general $k$}\label{sec general k}
% For proving the Theorem for a general $k$ we need to state a Lemma similar to the Lemma \ref{Lemma k=2}.

The proof of Theorem \ref{thm1} follows similar to the %proof of the 
special case $k=2$ once we state a general version of the Lemma \ref{Lemma k=2}. Recall that Lemma \ref{Lemma k=2} gave a bound on $ \bigg|E[f''(\varepsilon W+\varepsilon y)]-\mu E[f'(\varepsilon W+\varepsilon y)]-\varepsilon f''(\varepsilon y)
     +\mu f'(\varepsilon y)\bigg|$ in terms of $\varepsilon^2\|f^{(4)}\| $. For proving the general case, %theorem for $k \geq 2$, 
     we need to bound the same quantity using $\varepsilon^k \|f^{(k+2)}\| $ as given by the following lemma whose proof is presented in Appendix~\ref{secB1}.
\begin{lemma}\label{lemma general k}
    Let $k \in \mathbb N$. For any smooth $f$ there exists $b_{k,1},b_{k,2}, \ldots, b_{k,2^{k-1}}$ and $d_{k,1},d_{k,2}, \ldots, d_{k,2^{k-1}}$ with $b_{k,j}$ and $d_{k,j}$ constants and $3 \leq d_{k,1},d_{k,2}, \ldots, d_{k,2^{k-1}} \leq k+2$ such that 
    \begin{align*}
        \bigg|\mathbb E[f''(\varepsilon W+\varepsilon y)]-\mu \mathbb E[f'(\varepsilon W+\varepsilon y)]-\varepsilon f''(\varepsilon y)
     +\mu f'(\varepsilon y)\bigg|\leq \varepsilon^k \|f^{(k+2)}\| \bigg|\sum_{j=1}^{2^{k-1}}b_{k,j}  \left[ \int_0^\infty s^{d_{k,j}}dF(s)\right]\bigg|,
    \end{align*}
where $F$ denotes the distribution function of $S$.
\end{lemma}
Note that the only the derivative of order $(k+2)$ appear on the RHS. %We provide the proof of Lemma \ref{lemma general k} in Appendix~\ref{secB1}.
\begin{remark}\label{rem:induction}
The proof of Lemma~\ref{lemma general k} is considerably more intricate than that of Lemma~\ref{Lemma k=2}. In the proof of Lemma~\ref{Lemma k=2}, we expanded the generator from Lemma~\ref{lemma1} using a Taylor series and, by choosing an appropriate test function, expressed the third derivative in terms of the second and fourth derivatives. However, this approach is insufficient %alone does not suffice 
for Lemma~\ref{lemma general k}. To obtain a bound of order $\varepsilon^k$, the generator in Lemma~\ref{lemma1} must be expanded up to $(k+1)$ terms, producing derivatives of orders~$1$ through $k+2$. Our goal is then to express the resulting error term solely in terms of the $(k+2)^{\text{nd}}$ derivative.

We accomplish this inductively: by selecting carefully constructed test functions at each stage of the induction, we express the $i^{\text{th}}$ derivative in terms of the $(i-1)^{\text{th}}$ and $(k+2)^{\text{nd}}$ derivatives for all $3 \leq i \leq k+1$. This recursive elimination yields an error term that depends only on the $(k+2)^{\text{nd}}$ derivative of the Stein solution in~\eqref{steq}, which can then be uniformly bounded using Lemma~\ref{lemma3}.   
\end{remark}

Now we have the ingredients to prove Theorem~\ref{thm1} for general $k$. The proof proceeds similar to %in the same manner as 
the case $k=2$, with the distinction that the required bound is now of order $\varepsilon^k$, Consequently, instead of truncating the Taylor expansion of $f_h$ around 0 after the third term, as in \eqref{stg20.1}, we expand it up to $k+1$ terms. %By a
Applying the moment matching condition together with Lemma \ref{lemma general k}, we obtain an error bound that depends on $f_h$ only through its $(k+1)^{th}$ and $(k+2)^{nd}$ derivative and on $\varepsilon^k$.
\begin{proof}[Proof of Theorem \ref{thm1}]
As before, our goal is to bound the quantity
\begin{align}\label{bd}
    \bigg|-\mu \mathbb E[f_h'(\varepsilon W+\varepsilon S)]+ \mathbb E[f_h''(\varepsilon W+\varepsilon S)]+\mu f_h'(0)\bigg|.
\end{align}
We look at each of the terms separately starting with $ \mathbb E[f_h''(\varepsilon W+\varepsilon S)]$
\begin{align}\label{stg19}
\mathbb E[f_h''(\varepsilon W + \varepsilon S)] &=\int_0^\infty \mathbb E[f_h''(\varepsilon W+\varepsilon y)]g(y)dy\\ \nonumber
&=\int_0^\infty \left[\mathbb E[f_h''(\varepsilon W+\varepsilon y)]-\mu \mathbb E[f_h'(\varepsilon W+\varepsilon y)]-\varepsilon f_h''(\varepsilon y)
     +\mu f_h'(\varepsilon y)\ \right] g(y)dy\\ 
   +&\int_0^\infty\mu \mathbb E[f_h'(\varepsilon W+\varepsilon y)]g(y)dy+\int_0^\infty\varepsilon f_h''(\varepsilon y)g(y)dy-\int_0^\infty \mu f_h'(\varepsilon y)g(y)dy .
   % \nonumber \\
   % &+\int_0^\infty \left( \sum_{j=1}^{2^{k-1}}b_{k,j} \varepsilon^{k} \mathbb E \left[ \int_0^\infty s^{d_{k,j}}f^{(k+2)}(\eta ) dF(s)\right]\right) g(y)dy
\end{align}
The first integral will be bounded using Lemma \ref{lemma general k}. Observe that the second integral equals  $\mu \mathbb E[f_h'(\varepsilon W+\varepsilon S)]$ which also appears in~\eqref{bd}. Thus, it remains only to evaluate the last two integrals. We begin with an integration by parts.
\begin{align*}
    \int_0^\infty\varepsilon f_h''(\varepsilon y)g(y)dy&=\varepsilon \left[ \frac{f_h'(\varepsilon y)}{\varepsilon}g(y)\bigg|_0^\infty-\int_0^\infty \frac{f_h'(\varepsilon y)}{\varepsilon }g'(y)dy\right]\\
    &=- f_h'(0)g(0)-\int_0^\infty f_h'(\varepsilon y)g'(y)dy.
\end{align*}
Thus, the difference of the final two integrals in~\eqref{stg19} equals 
\begin{align}\label{stg20}
& \int_0^\infty\varepsilon f_h''(\varepsilon y)g(y)dy-\int_0^\infty \mu f_h'(\varepsilon y)g(y)dy \nonumber \\
&=- f_h'(0)g(0)-\left[\int_0^\infty f_h'(\varepsilon y)g'(y)dy+\int_0^\infty \mu f_h'(\varepsilon y)g(y)dy\right]\nonumber\\
&=-f_h'(0)g(0)-\left[\int_0^\infty f_h'(\varepsilon y)(g'(y)+\mu g(y))dy \right] \nonumber\\
&\stackrel{(a)}{=}-f_h'(0)g(0) \nonumber\\
&-\left[\int_0^\infty \left[f_h'(0)+\varepsilon y f_h''(0)+\frac{\varepsilon^2y^2}{2}f_h'''(0)+\ldots+\frac{\varepsilon^{k-1}y^{k-1}}{(k-1)!}f_h^{(k)}(0)+\frac{\varepsilon^{k}y^{k}}{(k)!}f_h^{(k+1)}(\xi)\right](g'(y)+\mu g(y))dy \right] \nonumber\\
&=-f_h'(0)g(0) -\int_0^\infty \frac{\varepsilon^k y^k}{k!}f_h^{(k+1)}(\xi)(g'(y)+\mu g(y))dy \nonumber\\
&-\left[\int_0^\infty \left[f_h'(0)+\varepsilon y f_h''(0)+\frac{\varepsilon^2y^2}{2}f_h'''(0)+\ldots+\frac{\varepsilon^{k-1}y^{k-1}}{(k-1)!}f_h^{(k)}(0)\right](g'(y)+\mu g(y))dy \right],
\end{align}
where in $(a)$ we have Taylor expanded $f_h$ around 0. We now examine each of the $k$ integrals in the last term of~\eqref{stg20} individually. We begin with the first integral. We have
\begin{align}\label{stg21}
    \int_0^\infty f_h'(0)(g'(y)+\mu g(y))dy=-f_h'(0)g(0)+\mu f_h'(0),
\end{align}
since $g$ is a pdf with non-negative support and  $\lim_{y \to \infty}g(y)=0$.
For $1 \leq i \leq k-1$,
\begin{align}\label{stg22}
    &\int_0^\infty \frac{\varepsilon^i y^i}{i!} f_h^{(i+1)}(0)(g'(y)+\mu g(y))dy \nonumber\\
    &\stackrel{(a)}{=}\frac{\varepsilon^i}{i!} f_h^{(i+1)}(0)\left[y^ig(y)\bigg|_0^\infty-\int_0^\infty iy^{i-1} g(y)dy \right]+\frac{\varepsilon^i}{i!} f_h^{(i+1)}(0)\mu \int_0^\infty y^ig(y)dy \nonumber \\
    &\stackrel{(b)}{=}-\frac{\varepsilon^i}{(i-1)!} f_h^{(i+1)}(0).\frac{(i-1)!}{\mu^{i-1}}+\frac{\varepsilon^i}{i!} f_h^{(i+1)}(0)\mu \frac{i!}{\mu^i} \nonumber\\
    &=0,
\end{align}
where $(a)$ follows using integration by parts and in $(b)$ we have used the Assumption $\mathbb E[S^i]=\mathbb E[Z^i]=\frac{i!}{\mu^i}$. Substituting~\eqref{stg21} and~\eqref{stg22} into~\eqref{stg20} yields the desired expression
\begin{align}\label{stg18.1}
   & \int_0^\infty\varepsilon f_h''(\varepsilon y)g(y)dy-\int_0^\infty \mu f_h'(\varepsilon y)g(y)dy =-\mu f_h'(0) -\int_0^\infty \frac{\varepsilon^k y^k}{k!}f_h^{(k+1)}(\xi)(g'(y)+\mu g(y))dy .
\end{align}
Putting \eqref{stg18.1} in \eqref{stg19},
\begin{align*}
    \mathbb E[f_h''(\varepsilon W+\varepsilon S)]&=\mu \mathbb E[f_h'(\varepsilon W+\varepsilon S)]-\mu f_h'(0) -\int_0^\infty \frac{\varepsilon^k y^k}{k!}f_h^{(k+1)}(\xi)(g'(y)+\mu g(y))dy \\
  &+\int_0^\infty \left[E[f_h''(\varepsilon W+\varepsilon y)]-\mu E[f_h'(\varepsilon W+\varepsilon y)]-\varepsilon f_h''(\varepsilon y)
     +\mu f_h'(\varepsilon y)\ \right] g(y)dy.
\end{align*}
Therefore,
\begin{align*}
    &\bigg|-\mu \mathbb E[f_h'(\varepsilon W+\varepsilon S)]+ \mathbb E[f_h''(\varepsilon W+\varepsilon S)]+\mu f_h'(0)\bigg|\\
    &\leq \int_0^\infty \bigg|\frac{\varepsilon^ky^k}{k!}f_h^{(k+1)}(\xi) (g'(y)+\mu g(y)) \bigg|dy\\
    &+\int_0^\infty\bigg|\mathbb E[f_h''(\varepsilon W+\varepsilon y)]-\mu E[f_h'(\varepsilon W+\varepsilon y)]-\varepsilon f_h''(\varepsilon y)
     +\mu f_h'(\varepsilon y)\bigg| g(y)dy\\
    &\stackrel{(a)} {\leq} \|f_h^{(k+1)}\|\frac{\varepsilon^k}{k!}\int_0^\infty|y^k(g'(y)+\mu g(y))|dy+\|f_h^{(k+2)}\| \varepsilon^k\int_0^\infty \bigg|\left( \sum_{j=1}^{2^{k-1}}b_{k,j}   \left[ \int_0^\infty s^{d_{k,j}} dF(s)\right]\right) g(y) \bigg|dy\\
    &\stackrel{(b)}{\leq}  \varepsilon^k \bigg[\frac{1}{\mu k!}C_1+2  \left( \sum_{j=1}^{2^{k-1}}b_{k,j}   \left[ \int_0^\infty s^{d_{k,j}} dF(s)\right]\right) \bigg]\\
    &\leq C_2 \varepsilon^k,
\end{align*}
where  $(a)$ follows from the Lemma \ref{lemma general k} and $(b)$ follows Lemma \ref{lemma3} as we have $\|h^{(k)}\| \leq 1 $ and therefore $\|f_h^{(k+1)}\| \leq \frac{1}{\mu}$ and $\|f_h^{(k+2)}\| \leq 2$. Therefore from \eqref{mn}
\begin{align*}
    d_{Zol,k}(\varepsilon(W+S),Z)\leq C_2\varepsilon^k.
\end{align*}
\end{proof}
\subsection{Proof of Preliminary Lemmas}\label{sec:preliminarylemmaproof}
In this subsection, we provide the proofs of Lemma \ref{lemma1} and Lemma~\ref{lemma3} which were stated earlier in Section \ref{sec:proofs}.
\begin{proof}[Proof of Lemma \ref{lemma1}]
    From the definition of generators we have
\begin{align*}
   G_{\varepsilon W_t+\varepsilon y}f(z)=\lim_{t \to 0} \frac{\mathbb Ef(\varepsilon W_t+\varepsilon y)-f(z)}{t},
\end{align*}
where $z=\varepsilon W_0+\varepsilon y$. Let $\varepsilon W_0=x$. Therefore,
\begin{align}\label{e1}
       G_{\varepsilon W_t+\varepsilon y}f(x+\varepsilon y)&=\lim_{t \to 0} \frac{\mathbb Ef(\varepsilon W_t+\varepsilon y)-f(x+\varepsilon y)}{t} \nonumber\\
       &=\mathbbm{1}_{\{x=0\}}\left[ \lim_{t \to 0} \frac{\mathbb Ef(\varepsilon W_t+\varepsilon y)-f(\varepsilon y)}{t}\right]+\mathbbm{1}_{\{x>0\}}\left[ \lim_{t \to 0} \frac{\mathbb Ef(\varepsilon W_t+\varepsilon y)-f(x+\varepsilon y)}{t}\right].
\end{align}
 We consider the two terms separately. We begin with the second term
\begin{align*}
    &\mathbbm{1}_{\{x>0\}}\left[ \lim_{t \to 0} \frac{\mathbb Ef(\varepsilon W_t+\varepsilon y)-f(x+\varepsilon y)}{t}\right]\\
    &=\mathbbm{1}_{\{x>0\}}\left[ \lim_{t \to 0}\left[\mathbbm{1}_{\{\frac{x}{\varepsilon}-t>0\}} \frac{\mathbb Ef(\varepsilon W_t+\varepsilon y)-f(x+\varepsilon y)}{t}+\mathbbm{1}_{\{\frac{x}{\varepsilon} -t\leq 0\}}\frac{\mathbb Ef(\varepsilon W_t+\varepsilon y)-f(x+\varepsilon y)}{t}\right]\right]\\
    &\stackrel{(a)}{=}\mathbbm{1}_{\{x>0\}}\left[ \lim_{t \to 0}\left[\mathbbm{1}_{\{\frac{x}{\varepsilon}-t>0\}} \frac{\mathbb Ef(\varepsilon W_t+\varepsilon y)-f(x+\varepsilon y)}{t}\right]\right]\\
    &\stackrel{(b)}{=}\mathbbm{1}_{\{x>0\}}\left[\lim_{t \to 0}\frac{\mathbb E f\left(\varepsilon(\frac{x}{\varepsilon}-t+\sum_{i=1}^{G(t)}S_i)+\varepsilon y\right)-f(x+\varepsilon y)}{t} \right].
\end{align*}
 where in $(a)$ we have used that for $x>0$, $\lim_{t \to 0}\mathbbm{1}_{\{\frac{x}{\varepsilon}-t \leq 0\}}=0$ and $(b)$ follows as when $\frac{x}{\varepsilon}-t >0$ we have that $W_0=\frac{x}{\varepsilon}>t$, that is, the task that came in at time $0$ is still waiting to be served. Therefore the total waiting time at time $t$ is $\frac{x}{\varepsilon}-t+\sum_{i=1}^{G(t)}S_i$ where  $G(t)$ is the number of arrivals between $0$ and $t$ and $S_i$ are the service times of the $i^{th}$ arrival. Therefore,
\begin{align*}
 &\mathbbm{1}_{\{x>0\}}\left[ \lim_{t \to 0} \frac{\mathbb Ef(\varepsilon W_t+\varepsilon y)-f(x+\varepsilon y)}{t}\right]\\
    &\stackrel{(a)}{=}\mathbbm{1}_{\{x>0\}}\left[\lim_{t \to 0}\sum_{n \in \mathbb{N}}\frac{\mathbb E f\left(\varepsilon(\frac{x}{\varepsilon}-t+\sum_{i=1}^{n}S_i)+\varepsilon y\right)-f(x+\varepsilon y)}{t}.P(G(t)=n)\right]\\
    &=\mathbbm{1}_{\{x>0\}}\bigg[\lim_{t \to 0}\frac{\mathbb E f\left(\varepsilon(\frac{x}{\varepsilon}-t)+\varepsilon y\right)-f(x+\varepsilon y)}{t}.P(G(t)=0)\\
    &+\lim_{t \to 0}\mathbb E \left[ f\left(\varepsilon(\frac{x}{\varepsilon}-t+S_1)+\varepsilon y\right)-f(x+\varepsilon y)\right].\frac{P(G(t)=1)}{t}\\
    &+\lim_{t \to 0}\sum_{n \geq 2}\mathbb E \left[ f\left(\varepsilon(\frac{x}{\varepsilon}-t+\sum_{i=1}^n S_i)+\varepsilon y\right)-f(x+\varepsilon y)\right].\frac{P(G(t)=n)}{t}\bigg]\\
    &\stackrel{(b)}{=}\mathbbm{1}_{\{x>0\}}\left[ -\varepsilon f'(x+\varepsilon y)+\lambda \mathbb E[f(x+\varepsilon S+\varepsilon y)-f(x+\varepsilon y)]+0\right],
\end{align*}
where in $(a)$ we have conditioned on $G(t)$ and in $(b)$ we have used that $P(G(t)=n)=\frac{e^{-\lambda t}(\lambda t)^n}{n!}$. In particular $\lim_{t \to 0}P(G(t)=0)=1$, $\lim_{t \to 0}\frac{P(G(t)=1)}{t}=\lambda$ and $\lim_{t \to 0}\frac{P(G(t)=n)}{t}=0$ for $n \geq 2$.

We next consider the first term in \eqref{e1}
\begin{align*}
    &\mathbbm{1}_{\{x=0\}}\left[ \lim_{t \to 0} \frac{\mathbb Ef(\varepsilon W_t+\varepsilon y)-f(\varepsilon y)}{t}\right].
    % &\stackrel{(a)}{=}\mathbbm{1}_{\{x=0\}}\left[\lim_{t \to 0}\frac{\mathbb E f\left(\varepsilon(max\{\sum_{i=1}^{G(t)}S_i-t,0\})+\varepsilon y \right)-f(\varepsilon y)}{t} \right]\\
    % &=\mathbbm{1}_{\{x=0\}}\bigg[P(G(t)=0)\frac{f(\varepsilon y)-f(\varepsilon y)}{t}+\lim_{t \to 
    % 0}\frac{P(G(t)=1)}{t}(\mathbb E f(\varepsilon S -t +\varepsilon y)-f(\varepsilon y))\\
    % &+\lim_{t\to 0} \frac{\sum_{n \geq 2}P(G(t)=n)}{t}\mathbb E(f(\varepsilon(\sum S)-t+\varepsilon y)-f(\varepsilon y))\bigg]\\
    % &\stackrel{(b)}{=}\mathbbm{1}_{\{x=0\}}\left[\lambda(\mathbb E f(\varepsilon S+\varepsilon y)-f(\varepsilon y) \right]\\
    % &=\mathbbm{1}_{\{x=0\}}\left[\lambda(\mathbb E f(x+\varepsilon S+\varepsilon y)-f(x+\varepsilon y) \right],\\
\end{align*}
When $x=0$ we have $W_0=0$. Let $G(t)$ be the number of arrivals in $(0,t]$. If $G(t)=0$ then $W_t=0$ and given $G(t)=1$, let $U \sim \text{Unif}(0,t]$ be the arrival time of the task. Therefore
\begin{align*}
&\mathbbm{1}_{\{x=0\}}\left[ \lim_{t \to 0} \frac{\mathbb{E}f(\varepsilon W_t+\varepsilon y)-f(\varepsilon y)}{t}\right] \\
&=\mathbbm{1}_{\{x=0\}}\Bigg[\lim_{t \to 0}
P(G(t)=0)\frac{f(\varepsilon y)-f(\varepsilon y)}{t}
+\lim_{t \to 0}\frac{P(G(t)=1)}{t}\mathbb{E}\left[f\big(\varepsilon(S-(t-U))^+ + \varepsilon y\big)-f(\varepsilon y)\right] \\
&+\lim_{t \to 0}\sum_{n\ge 2}\frac{P(G(t)=n)}{t}\mathbb{E}\left[f(\varepsilon W_t+\varepsilon y)-f(\varepsilon y)\right]\Bigg]\\
&\stackrel{(a)}{=}\mathbbm{1}_{\{x=0\}}\Bigg[\lim_{t \to 0}\frac{P(G(t)=1)}{t}\mathbb{E}\left[f\big(\varepsilon(S-(t-U))^+ + \varepsilon y\big)-f(\varepsilon y)\right]\Bigg]\\
&\stackrel{(b)}{=}\mathbbm{1}_{\{x=0\}}\left[\lambda(\mathbb E f(\varepsilon S+\varepsilon y)-f(\varepsilon y) \right]\\
&=\mathbbm{1}_{\{x=0\}}\left[\lambda(\mathbb E f(x+\varepsilon S+\varepsilon y)-f(x+\varepsilon y) \right],
\end{align*}
where  we have again used $\lim_{t \to 0}\frac{P(G(t)=n)}{t}=0$  for $n \geq 2$ in $(a)$, and $\lim_{t \to 0}\frac{P(G(t)=1)}{t}=\lambda$ in $(b)$. 
Therefore from \eqref{e1},
\begin{align}\label{stg25}
     G_{\varepsilon W_t+\varepsilon y}f(x+\varepsilon y)&=\lambda \mathbb E[f(x+\varepsilon S+\varepsilon y)-f(x+\varepsilon y)]-\varepsilon f'(x+\varepsilon y)\mathbbm{1}_{\{x>0\}} \nonumber\\
     &=\lambda \int_0^\infty (f(x+\varepsilon y +\varepsilon s)-f(x+\varepsilon y))dF(s)-\varepsilon f'(x+\varepsilon y)\mathbbm{1}_{\{x>0\}}.
\end{align}
\end{proof}
\begin{proof}[Proof of Lemma~\ref{lemma3}]
Since $f_h$ satisfies
\begin{align*}
    -\mu f_h'(x)+ f_h''(x)=h(x)-\mathbb E[h(Y)],
\end{align*}
we have for any $k\geq 2$,
\begin{align}\label{l3}
     -\mu f_h^{(k+1)}(x)+ f_h^{(k+2)}(x)=h^{(k)}(x).
\end{align}
A solution of \eqref{l3} is given by 
$f_h^{(k+1)}(x)=-e^{\mu x }\int_x^\infty h^{(k)}e^{-\mu t }dt.$
Hence,
\begin{align*}
    |f_h^{(k+1)}(x)|=\left|-e^{\mu x }\int_x^\infty h^{(k)}(t)e^{-\mu t }dt\right|
    \leq e^{\mu x }\int_x^\infty |h^{(k)}(t)|e^{-\mu t }dt
     \leq \frac{\|h^{(k)}\|}{\mu}.
\end{align*}
Finally, using \eqref{l3},
\begin{align*}
    | f_h^{(k+2)}(x)|=|h^{(k)}(x)+\mu f_h^{(k+1)}(x)|
     \leq \|h^{(k)}\|+\|h^{(k)}\|
    \leq 2\|h^{(k)}\|.
\end{align*}
This completes the proof. 
\end{proof}

\section{Alternative approach for  $k=2$}\label{sec:altk=2}
In this section, we present an alternative approach for the second order approximation. Our argument relies on the Lindley recursion representation of the stationary sojourn time.
Using this method, we obtain a bound of order $\varepsilon^2$ for $ d_{Zol,2}$ between the scaled stationary sojourn time distribution and $\mathrm{Exp}(\mu)$, but with a different constant compared to Theorem~\ref{thm1}. We state this result formally below.
\begin{theorem}\label{thm:alt}
Let $D$ denote the steady-state sojourn time of an $M/G/1$ queue and let $Z \sim \mathrm{Exp}(\mu)$. Suppose the service time distribution satisfies $\mathbb{E}[S^i] = \mathbb{E}[Z^i]$ for $i \in \{1,2,3\}$ and $\mathbb E[S^4] < \infty$. Then 
\begin{align*}   
    d_{Zol,2}(\varepsilon D,Z)\leq  M\varepsilon^2,
\end{align*}
where
\begin{align*}
    M&=\frac{6\left(1 - \mathbb E[e^{-\lambda S}] - \lambda \mathbb E[S e^{-\lambda S}]\right)}{\lambda^2 (\mathbb E[e^{-\lambda S}])^2}+\frac{2}{\lambda \mu} +\frac{4(1-\mathbb E[e^{-\lambda S}])}{\lambda^2\mathbb E[e^{-\lambda S}]}+\frac{1-\mathbb E[e^{-\lambda S}]}{\lambda\mu \mathbb E[e^{-\lambda S}]}+\frac{14}{\lambda^2}+\frac{4}{\mu^2}+\frac{16 \lambda \mu}{4!} \bigg(\mathbb E[S^4] +16\frac{4!}{\lambda^4}\bigg)
\end{align*}
\end{theorem}
Note that on  $\epsilon<\frac{1}{2}$, we have 
\begin{align*}
    \frac{1}{\lambda^2}\leq \frac{4}{\mu^2}, \ \frac{1}{\lambda \mu} \leq \frac{2}{\mu^2}, \ \frac{\mu}{\lambda^3} \leq \frac{8}{\mu^2},
\end{align*}
and by Jensens,
\begin{align*}
    \mathbb E[e^{-\lambda S}] \geq e^{-\lambda \mathbb E[S]}=e^{-\lambda/\mu} \geq e^{-1}.
\end{align*}
Also we have $0 \leq 1 - \mathbb E[e^{-\lambda S}] - \lambda \mathbb E[S e^{-\lambda S}] \leq 1$. Therefore,
\begin{align*}
    M \leq\frac{24}{\mu^2}e^2+\frac{4}{\mu^2}+\frac{16e}{\mu^2}+\frac{2e}{\mu^2}+\frac{56}{\mu^2}+\frac{4}{\mu^2}+\frac{2}{3}\mu^2 \mathbb E[S^4]+\frac{2048}{\mu^2}=\frac{24}{\mu^2}e^2+\frac{18e}{\mu^2}+\frac{2112}{\mu^2}+\frac{2}{3}\mu^2 \mathbb E[S^4].
\end{align*}
Before proceeding to the proof, we first establish the structural foundation of the argument, beginning with the Lindley recursion representation. Let $X_{k+1}$ be the interarrival times between $k^{th}$ and $(k+1)^{th}$ task, $S_k$ be the service time of the $k^{th}$ task and  $D_k$ be the sojourn time of the $k^{th}$ task. Then the Lindley recursion gives
\begin{align}\label{lindeq8}
    D_{k+1}&=(D_k-X_{k+1})^++S_{k+1 } =(D_k-X_{k+1})+U_k+S_{k+1},
\end{align}
where 
\begin{align*}
    U_k=\begin{cases}
        0 \quad  \ &X_{k+1}\leq D_k\\
        X_{k+1}-D_k & X_{k+1}>D_k
    \end{cases}.
\end{align*}
From the definition, $(D_{k+1}-S_{k+1})U_k=0$ and $(D_k-X_{k+1}+U_k)U_k=0$. 

Let $\delta=\mathbb E[X_k-S_k]$. By the strong memoryless property of $X_{k+1}$,
\begin{align*}
    P(X_{k+1}>D_{k}+s|X_{k+1}>D_k)=P(X_{k+1}>s), \quad s \ge 0.
\end{align*}
Hence, $$P(U_k>s|U_k>0)=P(X_{k+1}>s).$$ Let $p  \triangleq P(U_k=0)$. Then   
\begin{align}\label{l1}
      U_k \sim \begin{cases}
        0, & \text{with probability } p, \\[4pt]
        \mathrm{Exp}(\lambda), & \text{with probability } 1-p .
    \end{cases}
\end{align}
By stationarity, $\mathbb{E}[D_{k+1}] = \mathbb{E}[D_k]$, which from \eqref{lindeq8} implies
\begin{align}\label{lindeq6}
    \mathbb E[U_k]=\mathbb E[X_{k+1}-S_{k+1}]=\frac{1}{\lambda}-\frac{1}{\mu}=\delta.
\end{align}
From \eqref{l1},  we also have $\mathbb{E}[U_k] = (1-p)\tfrac{1}{\lambda}$, so
$p = \tfrac{\lambda}{\mu}$. Furthermore,
\begin{align*}
    \mathbb{E}[U_k^2] = \frac{2}{\lambda}\left(\frac{1}{\lambda}-\frac{1}{\mu}\right),
    \qquad
    \mathbb{E}[U_k^3] = \frac{6}{\lambda^2}\left(\frac{1}{\lambda}-\frac{1}{\mu}\right).
\end{align*}

Further from Remark \ref{remark1} we have
\begin{align}\label{lindeq7}
    \mathbb E[D_k]=\frac{1}{\mu-\lambda}=\frac{1}{\mu(1-\frac{\lambda}{\mu})}.
\end{align}

This establishes the structural foundation of the argument. To complete the proof, we apply a Taylor expansion of $f(\delta D_{k+1})$ around $\delta D_k$. The expansion produces terms involving $\mathbb E[f'(\delta D_k)]$, $\mathbb E[f''(\delta D_k)]$ and higher derivatives, along with mixed terms depending on $U_k,S_{k+1},X_{k+1}$.

By exploiting the independence structure, the moment-matching assumptions, and the derivations above we show that many lower-order contributions cancel and by bounding the remainder terms  yields 
%a Stein type identity of the form
a Lemma given below. Its complete  proof is provided in  Appendix~\ref{app:lindley}.

\begin{lemma}\label{lemma:alt}
 For any smooth $f$,  we have
    \begin{align}\label{eq:delta-bound}
     &\bigg|- \mathbb E[f'(\delta D_k)]+\frac{1}{\lambda \mu} \mathbb E[f''(\delta D_k)]+f'(0)\bigg|\leq \delta^2(M'\|f^{(3)}\|+M''\|f^{(4)}\|). 
\end{align}
where 
\begin{align*}
    M'=\frac{6\left(1 - \mathbb E[e^{-\lambda S}] - \lambda \mathbb E[S e^{-\lambda S}]\right)}{\lambda^2 (\mathbb E[e^{-\lambda S}])^2}+\frac{2}{\lambda \mu} +\frac{4(1-\mathbb E[e^{-\lambda S}])}{\lambda^2\mathbb E[e^{-\lambda S}]}+\frac{1-\mathbb E[e^{-\lambda S}]}{\lambda\mu \mathbb E[e^{-\lambda S}]}+\frac{14}{\lambda^2}+\frac{4}{\mu^2}
\end{align*}
and
\begin{align*}
M''=\frac{8}{4!} \bigg(\mathbb E[S^4] +16\frac{4!}{\lambda^4}\bigg).
\end{align*}
\end{lemma}
 We now proceed with the proof of Theorem \ref{thm:alt}.
 \begin{proof}[Proof of Theorem \ref{thm:alt}]
 Since $\delta=\frac{1}{\lambda}-\frac{1}{\mu}$, we have $ \varepsilon=\lambda \delta$. Define the  function
\begin{align*}
    \tilde f(x) := f_h(\lambda x).
\end{align*}
where $f_h$ is the solution of the Stein equation \eqref{steq}.
Then
\begin{align*}
    \tilde f'(x) = \lambda f_h'(\lambda x), \qquad
\tilde f''(x) = \lambda^2 f_h''(\lambda x),
\qquad
\|\tilde f^{(3)}\| = \lambda^3 \|f_h^{(3)}\|,
\qquad
\|\tilde f^{(4)}\| = \lambda^4 \|f_h^{(4)}\|.
\end{align*}
Applying Lemma \ref{lemma:alt} with $f=\tilde f$ gives
\begin{align*}
    \Bigl|-\lambda \mathbb{E}[f_h'(\varepsilon D_k)] + \frac{\lambda}{\mu}\mathbb{E}[f_h''(\varepsilon D_k)] + \lambda f_h'(0)\Bigr|
&\leq M' \lambda \varepsilon^2\|f_h^{(3)}\|+  M'' \lambda^2 \varepsilon^2\|f_h^{(4)}\| \\
%\end{align*}
%That is 
%\begin{align*}
\Bigl|-\mu\mathbb{E}[f_h'(\varepsilon D_k)] + \mathbb{E}[f_h''(\varepsilon D_k)] + \mu f_h'(0)\Bigr|
& \leq  M' \mu \varepsilon^2\|f_h^{(3)}\|+  M'' \lambda \mu \varepsilon^2\|f_h^{(4)}\|.
\end{align*}
Hence using Lemma \ref{lemma3} we have for the steady-state sojourn time $D$
\begin{align*}
    d_{Zol,2}(\varepsilon D, Z) \leq \varepsilon^2(M'+2M'' \lambda \mu).
\end{align*}
\end{proof}
\section{Conclusion and Future Directions}
In this work, we studied the prelimit behavior of the stationary sojourn time in an $M/G/1$ queue under heavy traffic. The sojourn time serves as a fundamental measure of system performance. We established that as more moments of the service-time distribution are matched with those of the exponential distribution, the Zolotarev distance between the scaled stationary sojourn time and the exponential distribution decays at increasingly higher rates. This result highlights a hierarchy of approximation accuracy that depends on moment matching.

Several directions remain open for future investigation:
\begin{itemize}
   \item \textbf{Refined target distributions.} In this paper we showed conditions under which the exponential distribution becomes an increasingly accurate approximation. A natural next step is to explore whether a refined target distribution, rather than the exponential, can provide a better fit. Such refinements may relax some of the assumptions on the service-time distribution and yield sharper approximations, in the spirit of the approach in~\cite{BAD}.
   \item \textbf{Beyond sojourn time.} Our analysis relied heavily on the fact that the sojourn time has a continuous stationary distribution. Extending higher-order approximations to other performance metrics, such as waiting times or queue lengths, would be an important generalization. However, these cases involve prelimit distributions with atoms, and controlling the distance between a continuous and a discrete distribution presents new challenges.  
   \item \textbf{Alternative proofs.} We provided an alternative approach for the $k=2$  case of Theorem~\ref{thm1} using the Lindley recursion representation of the waiting time in the Section \ref{sec:altk=2}. Extending this argument to general $k \in \mathbb{N}$ would give a complementary and potentially more direct proof of the main result.
\end{itemize}  
\section{Acknowledgment}
This work was partially supported by NSF grants EPCN-2144316, CMMI-2140534, 2127778 and 2348409.

\bibliographystyle{ACM-Reference-Format}
\bibliography{Bibdata}
\appendix
\section{Proof of Lemma \ref{lemma general k}}\label{secB1}
\begin{proof}[Proof of Lemma \ref{lemma general k} ]
    Recall that we proved Lemma \ref{Lemma k=2} using the following key steps:
\begin{enumerate}
    \item We Taylor expanded the generator in Lemma \ref{lemma1} up to three terms.
    \item We applied the generator to a carefully chosen function to write the third derivative in terms of the second derivative and fourth derivative. 
    This expression of the the third derivative then can be used in the Taylor expansion in Step 1  to get the induction hypothesis which involves only the first, second and fourth derivatives. 
\end{enumerate}
The same idea can be used to prove Lemma \eqref{lemma general k} with some additional steps as we discuss below. 
We use induction to prove Lemma \ref{lemma general k}. Let $i$ be the induction variable and  $P(i)$ be the induction hypothesis defined below
\textit{
\begin{align}
   \nonumber &P(i):\ \text{For any smooth $f$ there exists $b_{i,1},b_{i,2}, \ldots, b_{i,2^{i-1}}$ and $d_{i,1},d_{i,2}, \ldots, d_{i,2^{i-1}}$ }\\
    &\nonumber \text{with $b_{i,j}$ and $d_{i,j}$ constants and $3 \leq d_{i,1},d_{i,2}, \ldots, d_{i,2^{i-1}} \leq i+2$ such that }\\
%\end{align}
%\begin{align}
     &- \sum_{j=1}^{2^{i-1}}b_{i,j} \varepsilon^{i}  \|f^{(i+2)}\|  \left[ \int_0^\infty s^{d_{i,j}} dF(s)\right]  \nonumber\\
     & \leq E[f''(\varepsilon W+\varepsilon y)]-\mu E[f'(\varepsilon W+\varepsilon y)]-\varepsilon f''(\varepsilon y)+\mu f'(\varepsilon y) \nonumber\\
     &\leq \sum_{j=1}^{2^{i-1}}b_{i,j} \varepsilon^{i}  \|f^{(i+2)}\|  \left[ \int_0^\infty s^{d_{i,j}} dF(s)\right]    .
\end{align}}
Note that $P(2)$ corresponds to Lemma \ref{Lemma k=2}. Now  assume $P(i)$ is true $\forall \ 2 \leq i \leq k-1 $. We show this implies $P(k)$. Similar to showing $P(2)$ this involves the following steps;

% m  
\begin{enumerate}
    \item  We Taylor expand the generator in Lemma \ref{lemma1} up to $(k+1)$ terms.
    \item  As in the $P(2)$ case we apply the generator to a carefully chosen function to write the $(k+1)^{th}$ derivative in terms of the $k^{th}$ derivative and the $(k+2)^{nd}$ derivative. 
    \\

For proving $P(2)$ the second step gave the third derivative in terms of  the second and fourth derivative, and that was all that was required to prove $P(2)$. Now for proving $P(k)$ we will have to write $f^{(3)}, f^{(4)}, \ldots, f^{(k+1)}$ in terms of $f^{(2)}$ and $f^{(k+2)}$. For this we will need $k-2$ more steps. The general $r^{th}$ step is given as follows for $3 \leq r \leq k$.
\\
    \item [(r)]Use $P(r-1)$ with carefully chosen function to write $f^{(k+3-r)}$ in terms of $f^{(k+2-r)}$ and $f^{(k+2)}$. %\hfill\refstepcounter{equation}(\theequation)\label{step}
\end{enumerate}
To get a better understanding we give the details of the first three steps before stating and proving Claim \ref{claim1} which gives the expression at the end of a general $r^{(th)}$ step.

From the generator in Lemma \ref{lemma1} we have 
\begin{align}\label{e2}
    G_{\varepsilon W_t+\varepsilon y}f(x+\varepsilon y)&=\lambda \int_0^\infty (f(x+\varepsilon y +\varepsilon s)-f(x+\varepsilon y))dF(s)-\varepsilon f'(x+\varepsilon y)\mathbbm{1}_{\{x>0\}}
\end{align}

\textbf{Step 1: } 
Expanding \eqref{e2} up to $(k+1)$ terms
\begin{align}\label{eqstep1}
       &G_{\varepsilon W_t+\varepsilon y}f(x+\varepsilon y) \nonumber\\ &=\lambda \varepsilon \mathbb E[S]f'(x+\varepsilon y)+\frac{1}{2}\lambda \varepsilon^2 \mathbb E[S^2]f''(x+\varepsilon y)+\frac{1}{6}\lambda \varepsilon^3 \mathbb E[S^3]f'''(x+\varepsilon y)+ \ldots \nonumber\\
       &+\frac{1}{(k+1)!}\lambda \varepsilon^{(k+1)}\mathbb E[S^{(k+1)}]f^{(k+1)}(x+\varepsilon y)-\varepsilon f'(x+\varepsilon y)+\varepsilon f'(x+\varepsilon y)\mathbbm{1}_{\{x=0\}}+\varepsilon f'(x+\varepsilon y)\mathbbm{1}_{\{x<0\}}\nonumber\\
       &+\frac{1}{(k+2)!}\lambda \varepsilon^{(k+2)} \int_0^\infty s^{(k+2)} f^{(k+2)}(\eta)dF(s)\nonumber\\
       &\stackrel{(\leq)}{\geq}  \lambda \varepsilon \mathbb E[S]f'(x+\varepsilon y)+\frac{1}{2}\lambda \varepsilon^2 \mathbb E[S^2]f''(x+\varepsilon y)+\frac{1}{6}\lambda \varepsilon^3 \mathbb E[S^3]f'''(x+\varepsilon y)+ \ldots\nonumber\\
       &+\frac{1}{(k+1)!}\lambda \varepsilon^{(k+1)}\mathbb E[S^{(k+1)}]f^{(k+1)}(x+\varepsilon y)-\varepsilon f'(x+\varepsilon y)+\varepsilon f'(x+\varepsilon y)\mathbbm{1}_{\{x=0\}}+\varepsilon f'(x+\varepsilon y)\mathbbm{1}_{\{x<0\}}\nonumber\\
       &\stackrel{(+)}{-}\frac{1}{(k+2)!}\lambda \varepsilon^{(k+2)} \|f^{(k+2)}\|\int_0^\infty s^{(k+2)} dF(s).
\end{align}

Next, we want to use $ \mathbb E[G_{\varepsilon W_t+\varepsilon y}f(\varepsilon W+\varepsilon y) ]=0$, for the steady state waiting time $W$. We first show that $ \mathbb E[|G_{\varepsilon W_t+\varepsilon y}f(\varepsilon W+\varepsilon y)| ] < \infty$. We have 
\begin{align*}
    \bigg|\int_0^\infty (f(u+\epsilon s)-f(u)) dF(s) \bigg| \leq \sum_{j=1}^k \frac{\epsilon^j \mathbb E[S^j]}{j!}|f^{(j)}(u)|+\frac{\epsilon^{(k+1)}\mathbb E[S^{(k+1)}]}{(k+1)!}\|f^{(k+1)}\|.
\end{align*} 
Now fix $j\in\{1,\dots,k\}$ and set $m := k+1-j$. By Taylor’s theorem applied to $f^{(j)}$ about $0$ up to order $m-1$,
\begin{align*}  
    f^{(j)}(u)
    = \sum_{r=0}^{m-1} \frac{f^{(j+r)}(0)}{r!}u^r
      + \frac{f^{(k+1)}(\xi)}{m!}u^m
\end{align*}
for some $\xi \in (0,u)$
Therefore,
\begin{align*}
      |f^{(j)}(u)|
    \leq \sum_{r=0}^{m-1} \frac{|f^{(j+r)}(0)|}{r!}|u|^r
         + \frac{\|f^{(k+1)}\|}{m!}|u|^{m}
    \leq M_j\big(1+|u|^{m}\big)
    = M_j\big(1+|u|^{k+1-j}\big),
\end{align*}
where $M_j$ is a constant and we have used that the $(k+1)^{th}$ derivative is bounded.
That is for $j\in\{1,\dots,k\}$
\begin{align*}
    \mathbb E[|f^{(j)}(\varepsilon W+\varepsilon y)|] \leq M_j\big(1+\mathbb E[(\varepsilon W+\varepsilon y
    )^{k+1-j}]\big) \leq M_j'\big(1+\varepsilon^{k+1-j}\mathbb E[ W^{k+1-j}]\big) < \infty
\end{align*}
where the last inequality follows from \eqref{wtrecursion}. Hence we have
\begin{align}\label{generatorbound}
      \mathbb E[|G_{\varepsilon W_t+\varepsilon y}f(\varepsilon W+\varepsilon y)| ] &\leq \lambda \bigg[ \sum_{j=1}^k \frac{\epsilon^j \mathbb E[S^j]}{j!}\mathbb E[|f^{(j)}(\varepsilon W+\varepsilon y)|]+\frac{\epsilon^{(k+1)}\mathbb E[S^{(k+1)}]}{(k+1)!}\|f^{(k+1)}\|\bigg]\\
      &+\epsilon \mathbb E[  |f'(\varepsilon W+\varepsilon y)|]< \infty
\end{align}

Therefore for the steady state waiting time $W$, $ \mathbb E[G_{\varepsilon W_t+\varepsilon y}f(\varepsilon W+\varepsilon y) ]=0$, and  we have using $P(W=0)=\varepsilon$ and  $P(W<0)=0$ in \eqref{eqstep1},
\begin{align*}
    0 &\stackrel{(\leq)}{\geq} \lambda \varepsilon \mathbb E[S]\mathbb E[f'(\varepsilon W+\varepsilon y)]+\frac{1}{2}\lambda \varepsilon^2 \mathbb E[S^2]\mathbb E[f''(\varepsilon W+\varepsilon y)]+\frac{1}{6}\lambda \varepsilon^3 \mathbb E[S^3]\mathbb E[f'''(\varepsilon W+\varepsilon y)]+ \ldots\\
    &+\frac{1}{(k+1)!}\lambda \varepsilon^{(k+1)}\mathbb E[S^{(k+1)}]\mathbb E[f^{(k+1)}(\varepsilon W+\varepsilon y)]\\
       &-\varepsilon \mathbb E[f'(\varepsilon W+\varepsilon y)]+\varepsilon \mathbb E [f'(\varepsilon y)\mathbbm{1}_{\{x=0\}}]\stackrel{(+)}{-}\frac{1}{(k+2)!}\lambda \varepsilon^{(k+2)} \|f^{(k+2)}\| \left[\int_0^\infty s^{(k+2)} dF(s)\right]\\
       &=\mathbb \varepsilon E[f'(\varepsilon W+\varepsilon y)]\left(\frac{\lambda}{\mu}-1 \right)+\frac{1}{\mu^2}\lambda \varepsilon^2 \mathbb E[f''(\varepsilon W+\varepsilon y)]+\frac{1}{\mu^3}\lambda \varepsilon^3 \mathbb E[f'''(\varepsilon W+\varepsilon y)]+ \ldots\\
    &+\frac{1}{\mu^{(k+1)}}\lambda \varepsilon^{(k+1)}\mathbb E[f^{(k+1)}(\varepsilon W+\varepsilon y)]\\
       &+\varepsilon^2 f'(\varepsilon y)\stackrel{(+)}{-}\frac{1}{(k+2)!}\lambda \varepsilon^{(k+2)} \|f^{(k+2)}\| \left[\int_0^\infty s^{(k+2)} dF(s)\right]\\
        &=-\varepsilon^2\mathbb E[f'(\varepsilon W+\varepsilon y)]+\frac{1}{\mu^2}\lambda \varepsilon^2 \mathbb E[f''(\varepsilon W+\varepsilon y)]+\frac{1}{\mu^3}\lambda \varepsilon^3 \mathbb E[f'''(\varepsilon W+\varepsilon y)]+ \ldots\\
    &+\frac{1}{\mu^{(k+1)}}\lambda \varepsilon^{(k+1)}\mathbb E[f^{(k+1)}(\varepsilon W+\varepsilon y)]\\
       &+\varepsilon^2 f'(\varepsilon y)\stackrel{(+)}{-}\frac{1}{(k+2)!}\lambda \varepsilon^{(k+2)}\|f^{(k+2)}\| \left[\int_0^\infty s^{(k+2)}dF(s)\right].\\
\end{align*}

%\textbf{Step 1:}
Let $b_{k,1}=\frac{\lambda}{(k+2)!}$ and $d_{k,1}=(k+2)$. Therefore,
\begin{align} \label{stg11}
 \nonumber   0 &\stackrel{(\leq)}{\geq} -\varepsilon^2\mathbb E[f'(\varepsilon W+\varepsilon y)]+\frac{1}{\mu^2}\lambda \varepsilon^2 \mathbb E[f''(\varepsilon W+\varepsilon y)]+\frac{1}{\mu^3}\lambda \varepsilon^3 \mathbb E[f'''(\varepsilon W+\varepsilon y)]+ \ldots\\
 \nonumber    &+\frac{1}{\mu^{(k+1)}}\lambda \varepsilon^{(k+1)}\mathbb E[f^{(k+1)}(\varepsilon W+\varepsilon y)]\\
       &+\varepsilon^2 f'(\varepsilon y)\stackrel{(+)}{-}b_{k,1} \varepsilon^{(k+2)} \|f^{(k+2)}\| \left[\int_0^\infty s^{d_{k,2}} dF(s)\right].
\end{align}
\textbf{Step 2:} This step is similar to the Step 2 in the $i=2$ case. Again, from the generator in  Lemma \ref{lemma1} we have 
\begin{align*}
    G_{\varepsilon W_t+\varepsilon y}f(x+\varepsilon y)&=\lambda \int_0^\infty (f(x+\varepsilon y +\varepsilon s)-f(x+\varepsilon y))dF(s)-\varepsilon f'(x+\varepsilon y)\mathbbm{1}_{\{x>0\}}\\
  \nonumber   &=\lambda \varepsilon \mathbb E[S]f'(x+\varepsilon y)+\frac{1}{2}\lambda \varepsilon^2 \mathbb E[S^2]f''(x+\varepsilon y)-\varepsilon f'(x+\varepsilon y)+\varepsilon f'(x+\varepsilon y)\mathbbm{1}_{\{x=0\}}\\
  \nonumber   &+\varepsilon f'(x+\varepsilon y)\mathbbm{1}_{\{x<0\}} +\frac{1}{6}\lambda \varepsilon^3 \int_0^\infty s^3 f'''(\eta)dF(s).
\end{align*}
In order to write the $(k+1)^{th}$ derivative in terms of the $k^{th}$ derivative and $(k+2)^{nd}$ derivative  we apply this generator to $\frac{\frac{1}{(k+1)!}\lambda \varepsilon^{(k+1)} \mathbb E[S^{(k+1)}]}{\frac{1}{2}\lambda \varepsilon^2 \mathbb E[S^2]}f^{(k-1)}(x+\varepsilon y)$. That is,
\begin{align*}
  & G_{\varepsilon W+\varepsilon y}\frac{\frac{1}{(k+1)!}\lambda \varepsilon^{(k+1)} \mathbb E[S^{(k+1)}]}{\frac{1}{2}\lambda \varepsilon^2 \mathbb E[S^2]}f^{(k-1)}(x+\varepsilon y)\\
  &=\lambda \varepsilon \mathbb E[S].\frac{\frac{1}{(k+1)!}\lambda \varepsilon^{(k+1)} \mathbb E[S^{(k+1)}]}{\frac{1}{2}\lambda \varepsilon^2 \mathbb E[S^2]}f^{(k)}(x+\varepsilon y)+\frac{1}{2}\lambda \varepsilon^2 \mathbb E[S^2].\frac{\frac{1}{(k+1)!}\lambda \varepsilon^{(k+1)} \mathbb E[S^{(k+1)}]}{\frac{1}{2}\lambda \varepsilon^2 \mathbb E[S^2]}f^{(k+1)}(x+\varepsilon y)\\
    &-\varepsilon.\frac{\frac{1}{(k+1)!}\lambda \varepsilon^{(k+1)} \mathbb E[S^{(k+1)}]}{\frac{1}{2}\lambda \varepsilon^2 \mathbb E[S^2]} f^{(k)}(x+\varepsilon y)+\varepsilon \frac{\frac{1}{(k+1)!}\lambda \varepsilon^{(k+1)} \mathbb E[S^{(k+1)}]}{\frac{1}{2}\lambda \varepsilon^2 \mathbb E[S^2]}f^{(k)}(x+\varepsilon y)\mathbbm{1}_{\{x=0\}}\\
    &+\varepsilon \frac{\frac{1}{(k+1)!}\lambda \varepsilon^{(k+1)} \mathbb E[S^{(k+1)}]}{\frac{1}{2}\lambda \varepsilon^2 \mathbb E[S^2]}f^{(k)}(x+\varepsilon y)\mathbbm{1}_{\{x<0\}}+\frac{1}{6} \lambda \varepsilon^3 \frac{\frac{1}{(k+1)!}\lambda \varepsilon^{(k+1)} \mathbb E[S^{(k+1)}]}{\frac{1}{2}\lambda \varepsilon^2 \mathbb E[S^2]} \int_0^\infty s^3 f^{(k+2)}(\eta )dF(s)\\
    &=\varepsilon(\rho-1 ).\frac{\frac{1}{(k+1)!}\lambda \varepsilon^{(k+1)} \mathbb E[S^{(k+1)}]}{\frac{1}{2}\lambda \varepsilon^2 \mathbb E[S^2]}f^{(k)}(x+\varepsilon y)+\frac{1}{(k+1)!}\lambda \varepsilon^{(k+1)} \mathbb E[S^{(k+1)}]f^{(k+1)}(x+\varepsilon y)\\
    &+\varepsilon \frac{\frac{1}{(k+1)!}\lambda \varepsilon^{(k+1)} \mathbb E[S^{(k+1)}]}{\frac{1}{2}\lambda \varepsilon^2 \mathbb E[S^2]}  f^{(k)}(x+\varepsilon y)\mathbbm{1}_{\{x=0\}} +\frac{1}{6} \lambda \varepsilon^3 \frac{\frac{1}{(k+1)!}\lambda \varepsilon^{(k+1)} \mathbb E[S^{(k+1)}]}{\frac{1}{2}\lambda \varepsilon^2 \mathbb E[S^2]} \int_0^\infty s^3 f^{(k+2)}(\eta )dF(s)\\
    &+\varepsilon \frac{\frac{1}{(k+1)!}\lambda \varepsilon^{(k+1)} \mathbb E[S^{(k+1)}]}{\frac{1}{2}\lambda \varepsilon^2 \mathbb E[S^2]}f^{(k)}(x+\varepsilon y)\mathbbm{1}_{\{x<0\}}\\
    &=-\varepsilon^2\frac{\frac{1}{(k+1)!}\lambda \varepsilon^{(k+1)} \mathbb E[S^{(k+1)}]}{\frac{1}{2}\lambda \varepsilon^2 \mathbb E[S^2]}f^{(k)}(x+\varepsilon y)+\frac{1}{(k+1)!}\lambda \varepsilon^{(k+1)} \mathbb E[S^{(k+1)}]f^{(k+1)}(x+\varepsilon y)\\
    &+\varepsilon \frac{\frac{1}{(k+1)!}\lambda \varepsilon^{(k+1)} \mathbb E[S^{(k+1)}]}{\frac{1}{2}\lambda \varepsilon^2 \mathbb E[S^2]} f^{(k)}(x+\varepsilon y)\mathbbm{1}_{\{x=0\}}
    +\frac{1}{6} \lambda \varepsilon^3 \frac{\frac{1}{(k+1)!}\lambda \varepsilon^{(k+1)} \mathbb E[S^{(k+1)}]}{\frac{1}{2}\lambda \varepsilon^2 \mathbb E[S^2]} \int_0^\infty s^3 f^{(k+2)}(\eta )dF(s)\\
    &+\varepsilon \frac{\frac{1}{(k+1)!}\lambda \varepsilon^{(k+1)} \mathbb E[S^{(k+1)}]}{\frac{1}{2}\lambda \varepsilon^2 \mathbb E[S^2]}f^{(k)}(x+\varepsilon y)\mathbbm{1}_{\{x<0\}}\\
   & \stackrel{(\geq)} {\leq} -\varepsilon^2\frac{\frac{1}{(k+1)!}\lambda \varepsilon^{(k+1)} \mathbb E[S^{(k+1)}]}{\frac{1}{2}\lambda \varepsilon^2 \mathbb E[S^2]}f^{(k)}(x+\varepsilon y)+\frac{1}{(k+1)!}\lambda \varepsilon^{(k+1)} \mathbb E[S^{(k+1)}]f^{(k+1)}(x+\varepsilon y)\\
    &+\varepsilon \frac{\frac{1}{(k+1)!}\lambda \varepsilon^{(k+1)} \mathbb E[S^{(k+1)}]}{\frac{1}{2}\lambda \varepsilon^2 \mathbb E[S^2]} f^{(k)}(x+\varepsilon y)\mathbbm{1}_{\{x=0\}}
    \stackrel{(-)}{+}\frac{1}{6} \lambda \varepsilon^3 \frac{\frac{1}{(k+1)!}\lambda \varepsilon^{(k+1)} \mathbb E[S^{(k+1)}]}{\frac{1}{2}\lambda \varepsilon^2 \mathbb E[S^2]} \|f^{(k+2)}\| \int_0^\infty s^3dF(s)\\
    &+\varepsilon \frac{\frac{1}{(k+1)!}\lambda \varepsilon^{(k+1)} \mathbb E[S^{(k+1)}]}{\frac{1}{2}\lambda \varepsilon^2 \mathbb E[S^2]}f^{(k)}(x+\varepsilon y)\mathbbm{1}_{\{x<0\}}.
\end{align*}
Moreover since $ \mathbb E[|G_{\varepsilon W_t+\varepsilon y}f^{(k-1)}(\varepsilon W+\varepsilon y)| ] <\infty$
\begin{align*}
   \mathbb E\left[ G_{\varepsilon W+\varepsilon y}\frac{\frac{1}{(k+1)!}\lambda \varepsilon^{(k+1)} \mathbb E[S^{(k+1)}]}{\frac{1}{2}\lambda \varepsilon^2 \mathbb E[S^2]}f^{(k-1)}(\varepsilon W+\varepsilon y)\right]=0,
\end{align*}
we have
\begin{align*}
   & \frac{1}{(k+1)!}\lambda \varepsilon^{(k+1)} \mathbb E[S^{(k+1)}]\mathbb E[f^{(k+1)}(\varepsilon W+\varepsilon y)]\\
   &\stackrel{(\leq)}{\geq} \varepsilon^2\frac{\frac{1}{(k+1)!}\lambda \varepsilon^{(k+1)} \mathbb E[S^{(k+1)}]}{\frac{1}{2}\lambda \varepsilon^2 \mathbb E[S^2]}\mathbb E[f^{(k)}(\varepsilon W+\varepsilon y)]-\varepsilon^2 \frac{\frac{1}{(k+1)!}\lambda \varepsilon^{(k+1)} \mathbb E[S^{(k+1)}]}{\frac{1}{2}\lambda \varepsilon^2 \mathbb E[S^2]} f^{(k)}(\varepsilon y)\\
    &\stackrel{(+)}{-}\frac{1}{6} \lambda \varepsilon^3 \frac{\frac{1}{(k+1)!}\lambda \varepsilon^{(k+1)} \mathbb E[S^{(k+1)}]}{\frac{1}{2}\lambda \varepsilon^2 \mathbb E[S^2]} \|f^{(k+2)}\| \int_0^\infty s^3 dF(s),
\end{align*}
that is,
\begin{align}\label{stg10}
 \nonumber \nonumber   \frac{\lambda \varepsilon^{(k+1)}}{\mu^{(k+1)}}\mathbb E[f^{(k+1)}(\varepsilon W+\varepsilon y)] &\stackrel{(\leq)} {\geq} \frac{\varepsilon^{(k+1)}}{\mu^{(k-1)}}\mathbb E[f^{(k)}(\varepsilon W+\varepsilon y)]-\frac{\varepsilon^{(k+1)}}{\mu^{(k-1)}}f^{(k)}(\varepsilon y)\\
    &\stackrel{(+)}{-}\frac{1}{6} \lambda \varepsilon^{(k+2)} \frac{1}{\mu^{(k-1)}} \|f^{(k+2)}\|\int_0^\infty s^3 dF(s).
\end{align}
Let $b_{k,2}=\frac{\lambda}{6\mu^{(k-1)}}$ and $d_{k,2}=3$. Putting in \eqref{stg11}
\begin{align}\label{stg13}
 \nonumber   0 &\stackrel{(\leq)} {\geq} -\varepsilon^2\mathbb E[f'(\varepsilon W+\varepsilon y)]+\frac{1}{\mu^2}\lambda \varepsilon^2 \mathbb E[f''(\varepsilon W+\varepsilon y)]+\frac{1}{\mu^3}\lambda \varepsilon^3 \mathbb E[f'''(\varepsilon W+\varepsilon y)]+ \ldots\\
 \nonumber &+\frac{1}{\mu^{(k-1)}}\lambda \varepsilon^{(k-1)}\mathbb E[f^{(k-1)}(\varepsilon W+\varepsilon y)]+\frac{1}{\mu^{(k)}}\lambda \varepsilon^{(k)}\mathbb E[f^{(k)}(\varepsilon W+\varepsilon y)]\\
 \nonumber    &+\frac{\varepsilon^{(k+1)}}{\mu^{(k-1)}}\mathbb E[f^{(k)}(\varepsilon W+\varepsilon y)]-\frac{\varepsilon^{(k+1)}}{\mu^{(k-1)}}f^{(k)}(\varepsilon y)\stackrel{(+)}{-}b_{k,2}\varepsilon^{(k+2)} \|f^{(k+2)}\|\int_0^\infty s^{d_{k,2}}dF(s)\\
 \nonumber      &+\varepsilon^2 f'(\varepsilon y)\stackrel{(+)}{-}b_{k,1} \varepsilon^{(k+2)} \|f^{(k+2)}\| \left[\int_0^\infty s^{d_{k,2}} dF(s)\right]\\
  \nonumber     &=-\varepsilon^2\mathbb E[f'(\varepsilon W+\varepsilon y)]+\frac{1}{\mu^2}\lambda \varepsilon^2 \mathbb E[f''(\varepsilon W+\varepsilon y)]+\frac{1}{\mu^3}\lambda \varepsilon^3 \mathbb E[f'''(\varepsilon W+\varepsilon y)]+ \ldots\\
        \nonumber &+\frac{1}{\mu^{(k-1)}}\lambda \varepsilon^{(k-1)}\mathbb E[f^{(k-1)}(\varepsilon W+\varepsilon y)]+\frac{\varepsilon^k}{\mu^{(k-1)}}\left(\frac{\lambda}{\mu}+\varepsilon \right)E[f^{(k)}(\varepsilon W+\varepsilon y)]\\
   \nonumber     &-\frac{\varepsilon^{(k+1)}}{\mu^{(k-1)}}f^{(k)}(\varepsilon y)+\varepsilon^2 f'(\varepsilon y)\\
\nonumber        &\stackrel{(+)}{-}b_{k,1} \varepsilon^{(k+2)} \|f^{(k+2)}\| \left[\int_0^\infty s^{d_{k,2}} dF(s)\right]\stackrel{(+)}{-}b_{k,2}\varepsilon^{(k+2)} \|f^{(k+2)}\|\int_0^\infty s^{d_{k,2}} dF(s)\\
 \nonumber     &=-\varepsilon^2\mathbb E[f'(\varepsilon W+\varepsilon y)]+\frac{1}{\mu^2}\lambda \varepsilon^2 \mathbb E[f''(\varepsilon W+\varepsilon y)]+\frac{1}{\mu^3}\lambda \varepsilon^3 \mathbb E[f'''(\varepsilon W+\varepsilon y)]+ \ldots\\
        \nonumber &+\frac{1}{\mu^{(k-1)}}\lambda \varepsilon^{(k-1)}\mathbb E[f^{(k-1)}(\varepsilon W+\varepsilon y)]+\frac{\varepsilon^k}{\mu^{(k-1)}}E[f^{(k)}(\varepsilon W+\varepsilon y)]\\
   \nonumber     &-\frac{\varepsilon^{(k+1)}}{\mu^{(k-1)}}f^{(k)}(\varepsilon y)+\varepsilon^2 f'(\varepsilon y)\\
       &\stackrel{(+)}{-}b_{k,1} \varepsilon^{(k+2)} \|f^{(k+2)}\| \left[\int_0^\infty s^{d_{k,2}} dF(s)\right]\stackrel{(+)}{-}b_{k,2}\varepsilon^{(k+2)} \|f^{(k+2)}\|\int_0^\infty s^{d_{k,2}} dF(s).
\end{align}
\textbf{Step 3:} We use $P(2)$ with carefully chosen function to write $f^{(k)}$ in terms of $f^{(k-1)}$ and $f^{(k+2)}$. From $P(2)$ we have that 
\begin{align}\label{stg12}
    E[f''(\varepsilon W+\varepsilon y)]& \stackrel{(\leq)} {\geq} \mu E[f'(\varepsilon W+\varepsilon y)]+\varepsilon f''(\varepsilon y)-\mu f'(\varepsilon y)\stackrel{(+)}{-}\sum_{j=1}^{2}b_{2,j} \varepsilon^{2}\|f^{(4)}\| \left[ \int_0^\infty s^{d_{2,j}} dF(s)\right]  .
\end{align}
Let $u_{k,k-1}(x+\varepsilon y)=\frac{\frac{\varepsilon^k}{\mu^{k-1}}}{\frac{1}{\mu}}f^{(k-2)}(x+\varepsilon y)$. Using this $u_{k,k-1}$ in \eqref{stg12},
\begin{align*}
    \frac{\varepsilon^k}{\mu^{k-1}}\mathbb E[f^{(k)}(\varepsilon W+\varepsilon y)] &\stackrel{(\leq)}{\geq} \frac{\varepsilon^k}{\mu^{k-2}} \mathbb E[f^{(k-1)}(\varepsilon W+\varepsilon y)]+\frac{\varepsilon^{k+1}}{\mu^{k-1}} f^{(k)}(\varepsilon y)-\frac{\varepsilon^k}{\mu^{k-2}}f^{(k-1)}(\varepsilon y)\\
    &\stackrel{(+)}{-}\sum_{j=1}^{2} \frac{b_{2,j}}{\mu^{k-1}} \varepsilon^{k+2} \|f^{(k+2)}\| \left[ \int_0^\infty s^{d_{2,j}}dF(s)\right]  . 
\end{align*}
Let $b_{k,3}=\frac{b_{2,1}}{\mu^{k-1}}$, $b_{k,4}=\frac{b_{2,2}}{\mu^{k-1}}$ and $d_{k,3}=d_{2,1}$, $d_{k,4}=d_{2,2}$. Therefore from \eqref{stg13}
\begin{align}\label{stg14}
0 &\stackrel{(\leq)}{\geq} -\varepsilon^2\mathbb E[f'(\varepsilon W+\varepsilon y)]+\frac{1}{\mu^2}\lambda \varepsilon^2 \mathbb E[f''(\varepsilon W+\varepsilon y)]+\frac{1}{\mu^3}\lambda \varepsilon^3 \mathbb E[f'''(\varepsilon W+\varepsilon y)]+ \ldots \nonumber\\
 &+\frac{1}{\mu^{(k-2)}}\lambda \varepsilon^{(k-2)}\mathbb E[f^{(k-2)}(\varepsilon W+\varepsilon y)]+\frac{1}{\mu^{(k-1)}}\lambda \varepsilon^{(k-1)}\mathbb E[f^{(k-1)}(\varepsilon W+\varepsilon y)] \nonumber \\
 &+\frac{\varepsilon^k}{\mu^{k-2}} \mathbb E[f^{(k-1)}(\varepsilon W+\varepsilon y)]+\frac{\varepsilon^{k+1}}{\mu^{k-1}} f^{(k)}(\varepsilon y)-\frac{\varepsilon^k}{\mu^{k-2}}f^{(k-1)}(\varepsilon y)-\frac{\varepsilon^{(k+1)}}{\mu^{(k-1)}}f^{(k)}(\varepsilon y)+\varepsilon^2 f'(\varepsilon y) \nonumber\\
&\stackrel{(+)}{-}\sum_{j=1}^{2^{3-1}}b_{k,j} \varepsilon^{k+2} \|f^{(k+2)}\| \left[ \int_0^\infty s^{d_{k,j}} dF(s)\right] \nonumber \\
&=-\varepsilon^2\mathbb E[f'(\varepsilon W+\varepsilon y)]+\frac{1}{\mu^2}\lambda \varepsilon^2 \mathbb E[f''(\varepsilon W+\varepsilon y)]+\frac{1}{\mu^3}\lambda \varepsilon^3 \mathbb E[f'''(\varepsilon W+\varepsilon y)]+ \ldots \nonumber\\
&+\frac{1}{\mu^{(k-2)}}\lambda \varepsilon^{(k-2)}\mathbb E[f^{(k-2)}(\varepsilon W+\varepsilon y)]+\frac{\varepsilon^{k-1}}{\mu^{k-2}}\left( \frac{\lambda}{\mu}+\varepsilon\right)\mathbb E[f^{(k-1)}(\varepsilon W+\varepsilon y)] \nonumber \\
&-\frac{\varepsilon^k}{\mu^{k-2}}f^{(k-1)}(\varepsilon y)+\varepsilon^2 f'(\varepsilon y)\stackrel{(+)}{-}\sum_{j=1}^{2^{3-1}}b_{k,j} \varepsilon^{k+2} \|f^{(k+2)}\|\left[ \int_0^\infty s^{d_{k,j}} dF(s)\right] \nonumber \\
&=-\varepsilon^2\mathbb E[f'(\varepsilon W+\varepsilon y)]+\frac{1}{\mu^2}\lambda \varepsilon^2 \mathbb E[f''(\varepsilon W+\varepsilon y)]+\frac{1}{\mu^3}\lambda \varepsilon^3 \mathbb E[f'''(\varepsilon W+\varepsilon y)]+ \ldots \nonumber\\
&+\frac{1}{\mu^{(k-2)}}\lambda \varepsilon^{(k-2)}\mathbb E[f^{(k-2)}(\varepsilon W+\varepsilon y)]+\frac{\varepsilon^{k-1}}{\mu^{k-2}}\mathbb E[f^{(k-1)}(\varepsilon W+\varepsilon y)] \nonumber\\
&-\frac{\varepsilon^k}{\mu^{k-2}}f^{(k-1)}(\varepsilon y)+\varepsilon^2 f'(\varepsilon y)\stackrel{(+)}{-}\sum_{j=1}^{2^{3-1}}b_{k,j} \varepsilon^{k+2} \|f^{(k+2)}\| \left[ \int_0^\infty s^{d_{k,j}} dF(s)\right]  .
\end{align}
To get the expression after the $r^{(th)}$ step we state the following claim. The proof of the claim is provided later. 
\begin{claim}\label{claim1}
%After $r$ steps we have 
For $3 \leq r \leq k-1$. there exists $b_{k,1},b_{k,2}, \ldots, b_{k,2^{r-1}}$ and $3 \leq d_{k,1},d_{k,2}, \ldots, d_{k,2^{r-1}} \leq k+2$ such that 
\begin{align}\label{stg26}
    0 &\stackrel{(\leq)} {\geq} -\varepsilon^2\mathbb E[f'(\varepsilon W+\varepsilon y)]+\frac{1}{\mu^2}\lambda \varepsilon^2 \mathbb E[f''(\varepsilon W+\varepsilon y)]+\frac{1}{\mu^3}\lambda \varepsilon^3 \mathbb E[f'''(\varepsilon W+\varepsilon y)]+ \ldots \nonumber\\
        \nonumber &+\frac{1}{\mu^{(k-r+1)}}\lambda \varepsilon^{(k-r+1)}\mathbb E[f^{(k-r+1)}(\varepsilon W+\varepsilon y)]+\frac{\varepsilon^{k-r+2}}{\mu^{k-r+1}}\mathbb E[f^{(k-r+2)}(\varepsilon W+\varepsilon y)]\\
    &-\frac{\varepsilon^{k-r+3}}{\mu^{k-r+1}}f^{(k-r+2)}(\varepsilon y)+\varepsilon^2 f'(\varepsilon y) \stackrel{(+)}{-}\sum_{j=1}^{2^{r-1}}b_{k,j} \varepsilon^{k+2} \|f^{(k+2)}\|\left[ \int_0^\infty s^{d_{k,j}} dF(s)\right].  
\end{align}
\end{claim}
The final step gives $P(k)$ as given below.\\
\textbf{Step $k$:} For $r=(k-1)$  in \eqref{stg26}
\begin{align}\label{stg17}
     0 &\stackrel{(\leq)} {\geq} -\varepsilon^2\mathbb E[f'(\varepsilon W+\varepsilon y)]+\frac{1}{\mu^2}\lambda \varepsilon^2 \mathbb E[f''(\varepsilon W+\varepsilon y)] \nonumber\\
        &+\frac{\varepsilon^3}{\mu^2}\mathbb E[f^{(3)}(\varepsilon W+\varepsilon y)]-\frac{\varepsilon^{4}}{\mu^{2}}f^{(3)}(\varepsilon y)+\varepsilon^2 f'(\varepsilon y)\stackrel{(+)}{-}\sum_{j=1}^{2^{k-2}}b_{k,j} \varepsilon^{k+2} \|f^{(k+2)}\|\left[ \int_0^\infty s^{d_{k,j}}dF(s)\right]  .
\end{align}
Furthermore, from $P(k-1)$ we have
\begin{align}\label{stg16}
     E[f''(\varepsilon W+\varepsilon y)]& \stackrel{(\leq)}{\geq}\mu E[f'(\varepsilon W+\varepsilon y)]+\varepsilon f''(\varepsilon y)-\mu f'(\varepsilon y)\stackrel{(+)}-\sum_{j=1}^{2^{k-2}}b_{k-1,j} \varepsilon^{k-1} \|f^{(k+2)}\| \left[ \int_0^\infty s^{d_{k-1,j}}\right] .
\end{align}
Let $u_{3,2}(x+\varepsilon y)=\frac{\frac{\varepsilon^3}{\mu^2}}{\frac{1}{\mu}}f'(x+\varepsilon y)$. Using  this $u_{3,2}$ in \eqref{stg16}
\begin{align*}
    \frac{\varepsilon^3}{\mu^2}\mathbb E[f^{(3)}(\varepsilon W+\varepsilon y)] &\stackrel{(\leq)}{\geq} \frac{\varepsilon^3}{\mu} \mathbb E[f''(\varepsilon W+\varepsilon y)]+\frac{\varepsilon^4}{\mu^2}f^{(3)}(\varepsilon y)-\frac{\varepsilon^3}{\mu}f''(\varepsilon y)\\
    &\stackrel{(+)}{-}\sum_{j=1}^{2^{k-2}} \frac{b_{k-1,j}}{\mu^2} \varepsilon^{k+2} \|f^{(k+2)}\| \left[ \int_0^\infty s^{d_{k-1,j}}dF(s)\right] .
\end{align*}
Let $b_{k,2^{(k-2)}+i}=\frac{b_{k-1,i}}{\mu}$ and $d_{k, 2^{(k-2)}+i}=d_{k-1,i}$ $\forall 1 \leq i \leq 2^{(k-2)}$. Therefore from \eqref{stg17},
\begin{align*}
     0 &\stackrel{(\leq)}{\geq} -\varepsilon^2\mathbb E[f'(\varepsilon W+\varepsilon y)]+\frac{1}{\mu^2}\lambda \varepsilon^2 \mathbb E[f''(\varepsilon W+\varepsilon y)]+ \frac{\varepsilon^3}{\mu} \mathbb E[f''(\varepsilon W+\varepsilon y)]+\frac{\varepsilon^4}{\mu^2}f^{(3)}(\varepsilon y)-\frac{\varepsilon^3}{\mu}f''(\varepsilon y)\nonumber\\
        &-\frac{\varepsilon^{4}}{\mu^{2}}f^{(3)}(\varepsilon y)+\varepsilon^2 f'(\varepsilon y)\stackrel{(+)}{-}\sum_{j=1}^{2^{k-1}} \frac{b_{k,j}}{\mu} \varepsilon^{k+2} \|f^{(k+2)}\|\left[ \int_0^\infty s^{d_{k,j}}dF(s)\right]\\
        &=-\varepsilon^2\mathbb E[f'(\varepsilon W+\varepsilon y)]+\frac{\varepsilon^2}{\mu} \left( \frac{\lambda}{\mu}+\varepsilon \right) \mathbb E[f''(\varepsilon W+\varepsilon y)]-\frac{\varepsilon^3}{\mu}f''(\varepsilon y)\nonumber\\
        &+\varepsilon^2 f'(\varepsilon y)\stackrel{(+)}{-}\sum_{j=1}^{2^{k-1}} \frac{b_{k,j}}{\mu} \varepsilon^{k+2} \|f^{(k+2)}\| \left[ \int_0^\infty s^{d_{k,j}} dF(s)\right]\\
        &=-\varepsilon^2\mathbb E[f'(\varepsilon W+\varepsilon y)]+\frac{\varepsilon^2}{\mu}  \mathbb E[f''(\varepsilon W+\varepsilon y)]-\frac{\varepsilon^3}{\mu}f''(\varepsilon y)\nonumber\\
        &+\varepsilon^2 f'(\varepsilon y)\stackrel{(+)}{-}\sum_{j=1}^{2^{k-1}} \frac{b_{k,j}}{\mu} \varepsilon^{k+2} \|f^{(k+2)}\|\left[ \int_0^\infty s^{d_{k,j}} dF(s)\right].
\end{align*}
Hence,
\begin{align*}
   \mathbb  E[f''(\varepsilon W+\varepsilon y)]& \stackrel{(\leq)}{\geq} \mu \mathbb E[f'(\varepsilon W+\varepsilon y)]+\varepsilon f''(\varepsilon y)-\mu f'(\varepsilon y)\stackrel{(+)}{-}\sum_{j=1}^{2^{k-1}}b_{k,j} \varepsilon^{k} \|f^{(k+2)}\| \left[ \int_0^\infty s^{d_{k,j}}dF(s)\right].
\end{align*}
This proves $P(k)$.

It remains to prove Claim \ref{claim1} which we provide below.

\begin{proof}{(of Claim \ref{claim1}) }
We again prove by induction.
Equation \eqref{stg13} and  \eqref{stg14} imply the claim for $r=2$ and $r=3$ respectively. Assume that the claim is true for  $r-1$, that is
\begin{align}\label{stg15}
    0 &\stackrel{(\leq)} {\geq} -\varepsilon^2\mathbb E[f'(\varepsilon W+\varepsilon y)]+\frac{1}{\mu^2}\lambda \varepsilon^2 \mathbb E[f''(\varepsilon W+\varepsilon y)]+\frac{1}{\mu^3}\lambda \varepsilon^3 \mathbb E[f'''(\varepsilon W+\varepsilon y)]+ \ldots \nonumber\\
        \nonumber &+\frac{1}{\mu^{(k-r+2)}}\lambda \varepsilon^{(k-r+2)}\mathbb E[f^{(k-r+2)}(\varepsilon W+\varepsilon y)]+\frac{\varepsilon^{k-r+3}}{\mu^{k-r+2}}\mathbb E[f^{(k-r+3)}(\varepsilon W+\varepsilon y)]\\
    &-\frac{\varepsilon^{k-r+4}}{\mu^{k-r+2}}f^{(k-r+3)}(\varepsilon y)+\varepsilon^2 f'(\varepsilon y) \stackrel{(+)}{-}\sum_{j=1}^{2^{r-2}}b_{k,j} \varepsilon^{k+2} \|f^{(k+2)}\| \left[ \int_0^\infty s^{d_{k,j}}dF(s)\right]  .
\end{align}
From $P(r-1)$ we have,
\begin{align}\label{stg24}
    E[f''(\varepsilon W+\varepsilon y)]& \stackrel{(\leq)}{\geq} \mu E[f'(\varepsilon W+\varepsilon y)]+\varepsilon f''(\varepsilon y)-\mu f'(\varepsilon y) \stackrel{(+)}{-}\sum_{j=1}^{2^{r-2}}b_{r-1,j} \varepsilon^{r-1} \|f^{(k+2)}\|\left[ \int_0^\infty s^{d_{r-1,j}}dF(s)\right]  .
\end{align}
Let $u_{k-r+3,k-r+2}(x+\varepsilon y)=\frac{\frac{\varepsilon^{k-r+3}}{\mu^{k-r+2}}}{\frac{1}{\mu}}f^{(k-r+1)}(x+\varepsilon y)$. Using this in \eqref{stg24}
\begin{align*}
    \frac{\varepsilon^{k-r+3}}{\mu^{k-r+2}} \mathbb E[f^{(k-r+3)}(\varepsilon W+\varepsilon y)] &\stackrel{(\leq) } \geq \frac{\varepsilon^{k-r+3}}{\mu^{k-r+1}}\mathbb E[f^{(k-r+2)}(\varepsilon W+\varepsilon y)]+\frac{\varepsilon^{k-r+4}}{\mu^{k-r+2}}f^{(k-r+3)}(\varepsilon y)\\
    &-\frac{\varepsilon^{k-r+3}}{\mu^{k-r+1}}f^{(k-r+2)}(\varepsilon y)\stackrel{(+)}{-}\sum_{j=1}^{2^{r-2}} \frac{b_{r-1,j}}{\mu^{k-r+2}} \varepsilon^{k+2} \|f^{(k+2)}\| \left[ \int_0^\infty s^{d_{r-1,j}} dF(s)\right] .
\end{align*}
 Let $b_{k,2^{(r-2)}+i}=\frac{b_{r-1,i}}{\mu^{k-r+2}}$ and $d_{k,2^{(r-2)}+i}=d_{r-1,i}$ $\forall \ 1 \leq i \leq 2^{(r-2)}$. Putting these in \eqref{stg15}
 \begin{align*}
    0 &\stackrel{(\leq)} {\geq} -\varepsilon^2\mathbb E[f'(\varepsilon W+\varepsilon y)]+\frac{1}{\mu^2}\lambda \varepsilon^2 \mathbb E[f''(\varepsilon W+\varepsilon y)]+\frac{1}{\mu^3}\lambda \varepsilon^3 \mathbb E[f'''(\varepsilon W+\varepsilon y)]+ \ldots \nonumber\\
    \nonumber &+\frac{1}{\mu^{(k-r+1)}}\lambda \varepsilon^{(k-r+1)}\mathbb E[f^{(k-r+1)}(\varepsilon W+\varepsilon y)]+\frac{1}{\mu^{(k-r+2)}}\lambda \varepsilon^{(k-r+2)}\mathbb E[f^{(k-r+2)}(\varepsilon W+\varepsilon y)]\\
    &+\frac{\varepsilon^{k-r+3}}{\mu^{k-r+1}}\mathbb E[f^{(k-r+2)}(\varepsilon W+\varepsilon y)]+\frac{\varepsilon^{k-r+4}}{\mu^{k-r+2}}f^{(k-r+3)}(\varepsilon y)-\frac{\varepsilon^{k-r+3}}{\mu^{k-r+1}}f^{(k-r+2)}(\varepsilon y)\\
    &-\frac{\varepsilon^{k-r+4}}{\mu^{k-r+2}}f^{(k-r+3)}(\varepsilon y)+\varepsilon^2 f'(\varepsilon y)\stackrel{(+)}{-}\sum_{j=1}^{2^{r-1}}b_{k,j} \varepsilon^{k+2} \|f^{(k+2)}\|\left[ \int_0^\infty s^{d_{k,j}} dF(s)\right]\\
    &=-\varepsilon^2\mathbb E[f'(\varepsilon W+\varepsilon y)]+\frac{1}{\mu^2}\lambda \varepsilon^2 \mathbb E[f''(\varepsilon W+\varepsilon y)]+\frac{1}{\mu^3}\lambda \varepsilon^3 \mathbb E[f'''(\varepsilon W+\varepsilon y)]+ \ldots \nonumber\\
     &+\frac{1}{\mu^{(k-r+1)}}\lambda \varepsilon^{(k-r+1)}\mathbb E[f^{(k-r+1)}(\varepsilon W+\varepsilon y)]+\frac{\varepsilon^{k-r+2}}{\mu^{k-r+1}}\left( \frac{\lambda}{\mu}+\varepsilon\right)\mathbb E[f^{(k-r+2)}(\varepsilon W+\varepsilon y)]\\
    &-\frac{\varepsilon^{k-r+3}}{\mu^{k-r+1}}f^{(k-r+2)}(\varepsilon y)+\varepsilon^2 f'(\varepsilon y)\stackrel{(+)}{-}\sum_{j=1}^{2^{r-1}}b_{k,j} \varepsilon^{k+2}\|f^{(k+2)}\| \left[ \int_0^\infty s^{d_{k,j}} dF(s)\right]\\
     &=-\varepsilon^2\mathbb E[f'(\varepsilon W+\varepsilon y)]+\frac{1}{\mu^2}\lambda \varepsilon^2 \mathbb E[f''(\varepsilon W+\varepsilon y)]+\frac{1}{\mu^3}\lambda \varepsilon^3 \mathbb E[f'''(\varepsilon W+\varepsilon y)]+ \ldots \nonumber\\
     &+\frac{1}{\mu^{(k-r+1)}}\lambda \varepsilon^{(k-r+1)}\mathbb E[f^{(k-r+1)}(\varepsilon W+\varepsilon y)]+\frac{\varepsilon^{k-r+2}}{\mu^{k-r+1}}\mathbb E[f^{(k-r+2)}(\varepsilon W+\varepsilon y)]\\
    &-\frac{\varepsilon^{k-r+3}}{\mu^{k-r+1}}f^{(k-r+2)}(\varepsilon y)+\varepsilon^2 f'(\varepsilon y)\stackrel{(+)}{-}\sum_{j=1}^{2^{r-1}}b_{k,j} \varepsilon^{k+2} \|f^{(k+2)}\| \left[ \int_0^\infty s^{d_{k,j}} dF(s)\right].\\
\end{align*}
This proves the claim and completes the proof.
\end{proof}
\end{proof}
 \section{Proof of Lemma \ref{lemma:alt}}\label{app:lindley}
\begin{proof}[Proof of Lemma \ref{lemma:alt}]
Taylor expansion of $f(\delta D_{k+1})$ around $\delta D_k$ yields
\begin{align*}
    f(\delta D_{k+1})&=f(\delta D_k)+\delta f'(\delta D_k)(S_{k+1}-X_{k+1}+U_k)+\frac{1}{2}\delta^2f''(\delta D_k)(S_{k+1}-X_{k+1}+U_k)^2\\
    &+\frac{\delta^3}{6}f'''(\delta D_k)(S_{k+1}-X_{k+1}+U_k)^3+\frac{\delta^4}{4!}f^{(4)}(\delta \Tilde{Z})(S_{k+1}-X_{k+1}+U_k)^4,
\end{align*}
where $\tilde Z$ lies between $D_k$ and $D_{k+1}$. Taking expectation gives,
\begin{align*}
   \mathbb E [f(\delta D_{k+1})]&=\mathbb E[f(\delta D_k)]+\delta \mathbb E[f'(\delta D_k)(S_{k+1}-X_{k+1}+U_k)]+\frac{1}{2}\delta^2\mathbb E[f''(\delta D_k)(S_{k+1}-X_{k+1}+U_k)^2]\\
    &+\frac{1}{6}\delta^3\mathbb E[f'''(\delta D_k)(S_{k+1}-X_{k+1}+U_k)^3]+\frac{\delta^4}{4!}\mathbb E[f^{(4)}(\delta \tilde{Z})(S_{k+1}-X_{k+1}+U_k)^4].
\end{align*}
As $|f(x)| \leq |f(0)|+ \sum_{i=1}^k \frac{|f^{(i)}(0)|}{i!}|x|^i+\frac{|f^{(k+1)}(\xi)|}{(k+1)!}|x|^{(k+1)}$ for some $\xi \in (0,x)$,
\begin{align*}
      \mathbb E[|f(\delta D_k)|]\leq  |f(0)|+\sum_{i=1}^2\frac{|f^{(i)}(0)|}{i!}\mathbb E[\delta^iD_k^i]+\frac{\|f^{(3)}\|}{3!}\mathbb E[\delta^{3}D_k^{3}]< \infty,
\end{align*}
where we have used $\|f^{(3)}\|< \infty $ and \eqref{strecursion}. Thus $  \mathbb E[f(\delta D_k)]< \infty$ and we have $ \mathbb E [f(\delta D_{k+1})]=\mathbb E[f(\delta D_k)]$, that is,
\begin{align}
 \nonumber   0&=\delta \mathbb E[f'(\delta D_k)]\mathbb E[(S_{k+1}-X_{k+1})]+\delta \mathbb E[f'(\delta D_k)U_k]+\frac{1}{2}\delta^2 \mathbb E[f''(\delta D_k)] \mathbb E[(S_{k+1}-X_{k+1})^2] \nonumber\\
 &+\frac{1}{2}\delta^2 \mathbb E[f''(\delta D_k)U_k^2]+\delta^2 \mathbb E[f''(\delta D_k)U_k(S_{k+1}-X_{k+1})]+\frac{1}{6}\delta^3\mathbb E[f'''(\delta D_k)(S_{k+1}-X_{k+1}+U_k)^3] \nonumber\\
&\nonumber +\frac{\delta^4}{4!}\mathbb E[f^{(4)}(\delta \Tilde{Z})(S_{k+1}-X_{k+1}+U_k)^4] \\
\nonumber    &= -\delta^2 \mathbb E[f'(\delta D_k)] +\delta \mathbb E[f'(\delta D_k)U_k]+\frac{1}{2}\delta^2 \mathbb E[f''(\delta D_k)] \mathbb E[(S_{k+1}-X_{k+1})^2]+\frac{1}{2}\delta^2 \mathbb E[f''(\delta D_k)U_k^2]\\
\nonumber    &+\delta^2 \mathbb E[f''(\delta D_k)U_k(S_{k+1}-X_{k+1})]+\frac{1}{6}\delta^3\mathbb E[f'''(\delta D_k)(S_{k+1}-X_{k+1}+U_k)^3]\\
&+\frac{\delta^4}{4!}\mathbb E[f^{(4)}(\delta \Tilde{Z})(S_{k+1}-X_{k+1}+U_k)^4],
\end{align}
where we used the independence of $(X_{k+1},S_{k+1})$ from $D_k$. 
Using the moment matching assumptions in Theorem \ref{thm:alt} in the above yields
\begin{align}\label{lindeq5}
    0&=-\delta^2 \mathbb E[f'(\delta D_k)] +\delta \mathbb E[f'(\delta D_k)U_k]+\frac{1}{2}\delta^2 \mathbb E[f''(\delta D_k)]\left(\frac{2}{\mu^2}+\frac{2}{\lambda^2}-\frac{2}{\lambda \mu} \right)+\frac{1}{2}\delta^2 \mathbb E[f''(\delta D_k)U_k^2] \nonumber\\
\nonumber    &+\delta^2 \mathbb E[f''(\delta D_k)U_k(S_{k+1}-X_{k+1})]+\frac{1}{6}\delta^3\mathbb E[f'''(\delta D_k)(S_{k+1}-X_{k+1}+U_k)^3]\\
&+\frac{\delta^4}{4!}\mathbb E[f^{(4)}(\delta \Tilde{Z})(S_{k+1}-X_{k+1}+U_k)^4] \nonumber\\
% &=-\delta^2 \mathbb E[f'(\delta D_k)]+\frac{\delta^2}{\lambda \mu} \mathbb E[f''(\delta D_k)]-2\frac{\delta^2}{\lambda \mu} \mathbb E[f''(\delta D_k)] +\delta \mathbb E[f'(\delta D_k)U_k]+\delta^2 \mathbb E[f''(\delta D_k)]\left(\frac{1}{\mu^2}+\frac{1}{\lambda^2}\right) \nonumber\\
% &+\frac{1}{2}\delta^2 \mathbb E[f''(\delta D_k)U_k^2]+\delta^2 \mathbb E[f''(\delta D_k)U_k(S_{k+1}-X_{k+1})]+\frac{1}{6}\delta^3\mathbb E[f'''(\delta D_k)(S_{k+1}-X_{k+1}+U_k)^3] \nonumber\\
% &+\frac{\delta^4}{4!}\mathbb E[f^{(4)}(\delta \Tilde{Z})(S_{k+1}-X_{k+1}+U_k)^4]\\
&=-\delta^2 \mathbb E[f'(\delta D_k)]+\frac{\delta^2}{\lambda \mu} \mathbb E[f''(\delta D_k)] \nonumber\\
&\underbrace{+\delta \mathbb E[f'(\delta D_k)U_k]+\delta^2 \mathbb E[f''(\delta D_k)]\left(\frac{1}{\mu}-\frac{1}{\lambda}\right)^2+\frac{1}{2}\delta^2 \mathbb E[f''(\delta D_k)U_k^2]+\delta^2 \mathbb E[f''(\delta D_k)U_k(S_{k+1}-X_{k+1})]}_{\triangleq I} \nonumber\\
&+\frac{1}{6}\delta^3\mathbb E[f'''(\delta D_k)(S_{k+1}-X_{k+1})^3]+\frac{1}{6}\delta^3 \mathbb E[f'''(\delta D_k) [U_k^3+3(S_{k+1}-X_{k+1})^2U_k+3(S_{k+1}-X_{k+1})U_k^2]] \nonumber\\
&+\frac{\delta^4}{4!}\mathbb E[f^{(4)}(\delta \Tilde{Z})(S_{k+1}-X_{k+1}+U_k)^4].
\end{align}
Now
\begin{align}\label{ln10}
    I=&\delta \mathbb E[f'(\delta D_k)U_k]+\delta^2 \mathbb E[f''(\delta D_k)]\left(\frac{1}{\mu}-\frac{1}{\lambda}\right)^2 +\frac{1}{2}\delta^2 \mathbb E[f''(\delta D_k)U_k^2]+\delta^2 \mathbb E[f''(\delta D_k)U_k(S_{k+1}-X_{k+1})] \nonumber\\
    &=\delta \mathbb E\left[f'(0)U_k+\delta D_kU_kf''(0)+\frac{1}{2}\delta^2D_k^2U_kf'''(\eta_1)\right]\nonumber \\
    &+\delta^4 \mathbb E\left[f''(0)+\delta D_kf'''(\eta_2)\right]\nonumber \\
    &+\frac{1}{2}\delta^2 \mathbb E\left[f''(0)U_k^2+\delta D_k U_k^2f'''(\eta_3))\right]\nonumber \\
    &+\delta^2 \mathbb E\left[f''(0)U_k(S_{k+1}-X_{k+1})+\delta D_k U_k(S_{k+1}-X_{k+1})f'''(\eta_4)\right]\nonumber \\
    &=f'(0) \delta \mathbb E[U_k]+\delta^2f''(0)\underbrace{\left[ \mathbb E[D_kU_k]+\delta^2+\frac{1}{2}\mathbb E[U_k^2]+\mathbb E[U_k(S_{k+1}-X_{k+1})] \right]}_{\triangleq A}\nonumber \\
    &+\delta^3 \mathbb E \bigg[\frac{1}{2}D_k^2U_kf'''(\eta_1)+\delta^2 D_k f'''(\eta_2) +\frac{1}{2}D_kU_k^2f'''(\eta_3)+D_kU_k(S_{k+1}-X_{k+1})f'''(\eta_4)\bigg].
\end{align}

Since $D_kU_k=X_{k+1}U_k-U_k^2$ and $E[U_k^2]=\frac{2}{\lambda}\left(\frac{1}{\lambda}-\frac{1}{\mu}\right)$, we have
\begin{align}\label{lindeq1}
A&=-\frac{2}{\lambda \mu}+\mathbb E[D_kU_k]+\frac{1}{\mu^2}+\frac{1}{\lambda^2}+\frac{1}{2}\mathbb E[U_k^2]+\mathbb E[U_k(S_{k+1}-X_{k+1})]  \nonumber \\
&=-\frac{2}{\lambda \mu}+\mathbb E[X_{k+1}U_k]-\mathbb E[U_k^2]+\frac{1}{\mu^2}+\frac{1}{\lambda^2}+\frac{1}{2}\mathbb E[U_k^2]+\mathbb E[U_kS_{k+1}]-\mathbb E[U_kX_{k+1}] \nonumber \\
&\stackrel{(a)}{=}-\frac{2}{\lambda \mu}-\frac{1}{\lambda}\left(\frac{1}{\lambda}-\frac{1}{\mu}\right) +\frac{1}{\mu^2}+\frac{1}{\lambda^2}+\frac{1}{\mu}\left( \frac{1}{\lambda}-\frac{1}{\mu}\right) \nonumber \\
&=0,
\end{align}
where in (a) we used that $S_{k+1}$ and $U_k$ are independent. From \eqref{lindeq5}, \eqref{ln10} and \eqref{lindeq1} we have
\begin{align*}
     0&=-\delta^2 \mathbb E[f'(\delta D_k)]+\frac{\delta^2}{\lambda \mu} \mathbb E[f''(\delta D_k)]+f'(0) \delta^2
      \\
      &+\delta^3 \mathbb E \bigg[\frac{1}{2}D_k^2U_kf'''(\eta_1)+\delta^2 D_k f'''(\eta_2) +\frac{1}{2}D_kU_k^2f'''(\eta_3)+D_kU_k(S_{k+1}-X_{k+1})f'''(\eta_4)\bigg]\\
      &+\frac{1}{6}\delta^3\mathbb E[f'''(\delta D_k)]\mathbb E[(S_{k+1}-X_{k+1})^3]+\frac{1}{6}\delta^3 \mathbb E[f'''(\delta D_k) [U_k^3+3(S_{k+1}-X_{k+1})^2U_k+3(S_{k+1}-X_{k+1})U_k^2]] \\
      &+\frac{\delta^4}{4!}\mathbb E[f^{(4)}(\delta \Tilde{Z})(S_{k+1}-X_{k+1}+U_k)^4].
\end{align*}
Hence
\begin{align}\label{ln12}
    &\delta^2\bigg|- \mathbb E[f'(\delta D_k)]+\frac{1}{\lambda \mu} \mathbb E[f''(\delta D_k)]+f'(0)\bigg| \nonumber\\
    &\leq \underbrace{\bigg|\delta^3 \mathbb E \bigg[\frac{1}{2}D_k^2U_kf'''(\eta_1)+\delta^2 D_k f'''(\eta_2) +\frac{1}{2}D_kU_k^2f'''(\eta_3)+D_kU_k(S_{k+1}-X_{k+1})f'''(\eta_4)\bigg]\bigg|}_{\triangleq B} \nonumber\\
      &+\bigg|\underbrace{\frac{1}{6}\delta^3\mathbb E[f'''(\delta D_k)]\mathbb E[(S_{k+1}-X_{k+1})^3]\bigg|}_{\triangleq C}+\underbrace{\bigg|\frac{1}{6}\delta^3 \mathbb E[f'''(\delta D_k) [U_k^3+3(S_{k+1}-X_{k+1})^2U_k+3(S_{k+1}-X_{k+1})U_k^2]] \bigg|}_{\triangleq F} \nonumber\\
      &+\underbrace{\bigg|\frac{\delta^4}{4!}\mathbb E[f^{(4)}(\delta \Tilde{Z})(S_{k+1}-X_{k+1}+U_k)^4]\bigg|}_{\triangleq E}.
\end{align}
We have 
\begin{align}\label{lnCbound}
   C= \bigg| \frac{1}{6}\delta^3\mathbb E[f'''(\delta D_k)]\mathbb E[(S_{k+1}-X_{k+1})^3]\bigg| &\leq  \frac{1}{6}\delta^3 \|f'''\| \bigg|6 \left(\frac{1}{\mu^3}-\frac{1}{\mu^2\lambda}+\frac{1}{\mu\lambda^2} -\frac{1}{\lambda^3}\right)\bigg| \nonumber\\
    &= \frac{1}{6}\delta^3 \|f'''\| \bigg|6 \left(\frac{1}{\mu}-\frac{1}{\lambda}\right)\left(\frac{1}{\mu^2}+\frac{1}{\lambda^2}\right)\bigg| \nonumber\\
    &=\delta^4 \|f'''\|\left(\frac{1}{\mu^2}+\frac{1}{\lambda^2}\right).
\end{align}

And
\begin{align}\label{lnD}
   F&= \bigg|\frac{1}{6}\delta^3 \mathbb E[f'''(\delta D_k) [U_k^3+3(S_{k+1}-X_{k+1})^2U_k+3(S_{k+1}-X_{k+1})U_k^2]] \bigg| \nonumber\\
    &\leq \delta^3 \frac{1}{6}\| f'''\|\mathbb E\left[U_k^3+3(S_{k+1}-X_{k+1})^2U_k+3|(S_{k+1}-X_{k+1})|U_k^2\right] \nonumber\\
    &=\delta^3 \frac{1}{6}\| f'''\|\bigg[\frac{6}{\lambda^2}\left(\frac{1}{\lambda}-\frac{1}{\mu}\right)+3\mathbb E[(S_{k+1}-X_{k+1})^2U_k|U_k>0]P(U_k>0) \nonumber \\
    &+3 \mathbb E[|(S_{k+1}-X_{k+1})|U_k^2|U_k>0]P(U_k>0) \bigg] \nonumber\\
    &=\delta^4 \frac{1}{6}\| f'''\|\bigg[\frac{6}{\lambda^2}+3\lambda\underbrace{\mathbb E[(S_{k+1}-X_{k+1})^2U_k|U_k>0]}_{\triangleq F_1}+3\lambda\underbrace{ \mathbb E[|(S_{k+1}-X_{k+1})|U_k^2|U_k>0]}_{\triangleq F_2}\bigg].
\end{align}
First we evaluate the term $F_2$. We have on $\{U_k >0\}$, $X_{k+1}=D_k+U_k$. That is  $|(S_{k+1}-X_{k+1})|=|(S_{k+1}-D_k-U_k| \leq S_{k+1}+D_k+U_k $. Therefore 
\begin{align*}
   \mathbb E[|(S_{k+1}-X_{k+1})|U_k^2|U_k>0] \leq \mathbb E[S_{k+1}U_k^2|U_k>0] +\mathbb E[D_{k}U_k^2|U_k>0]+\mathbb E[U_k^3|U_k>0].
\end{align*}
From \eqref{l1}, we have
\begin{align*}
    \mathbb E[U_k^3|U_k>0] =\frac{6}{\lambda^3},
\end{align*}
and since $S_{k+1}$ is independent of $U_k$,
\begin{align*}
    \mathbb E[S_{k+1}U_k^2|U_k>0]=\mathbb E[S_{k+1}]\mathbb E[U_k^2|U_k>0]=\frac{2}{\mu \lambda^2}.
\end{align*}
% From \eqref{l1} we have that conditional on $U_k > 0$, the random variable $U_k = X_{k+1} - D_k$ is exponentially distributed with rate $\lambda$ and is independent of $D_k$ and hence
We have using the strong memoryless property
\begin{align*}
      \mathbb E[D_{k}U_k^2|U_k>0]=\mathbb E[\mathbb E[D_kU_k^2|D_k,U_k>0]|U_k>0]=\mathbb E[D_k\mathbb E[U_k^2|D_k,U_k>0]|U_k>0]=\frac{2}{\lambda^2}\mathbb E[D_{k}|U_k>0].
\end{align*}
% \begin{align*}
%     \mathbb E[D_{k}U_k^2|U_k>0]=\mathbb E[U_k^2|U_k>0]\mathbb E[D_{k}|U_k>0]=\frac{2}{\lambda^2}\mathbb E[D_{k}|U_k>0]
% \end{align*}
Further,
\begin{align}\label{ln2}
    \mathbb E[D_{k}|U_k>0]=\frac{\mathbb E[D_k\mathrm{1}(U_k>0)]}{P(U_k>0)}&=\frac{\mathbb E[D_k\mathrm{1}(X_{k+1}>D_k)]}{\mathbb E[1(X_{k+1}>D_k)]} \nonumber\\
    &=\frac{\mathbb E\left[ \mathbb E[D_k\mathrm{1}(X_{k+1}>D_k)|D_k]\right]}{\mathbb E\left[ \mathbb E[\mathrm{1}(X_{k+1}>D_k)|D_k]\right]} \nonumber\\
    &=\frac{\mathbb E\left[D_k \mathbb E[\mathrm{1}(X_{k+1}>D_k)|D_k]\right]}{\mathbb E\left[ \mathbb E[\mathrm{1}(X_{k+1}>D_k)|D_k]\right]} \nonumber\\
    &=\frac{\mathbb E[D_ke^{-\lambda D_k}]}{\mathbb E[e^{-\lambda D_k}]}.
\end{align}
Now from the PK formula of sojourn-time
\begin{align*}
   F(s) := \mathbb E[e^{- s D_k}]=\left(1-\frac{\lambda}{\mu}\right)\frac{s\mathbb E[e^{-s S}]}{s-\lambda+\lambda \mathbb E[e^{-sS}]}.
\end{align*}
Then 
\begin{align*}
    \mathbb E[e^{-\lambda D_k}]=F(\lambda)=\left(1-\frac{\lambda}{\mu}\right),
\end{align*}
and 
\begin{align}\label{ln1}
    \mathbb E[D_ke^{-\lambda D_k}]=-F'(s)|_{s=\lambda}.
\end{align}
Let $\tilde{S}(s)=E[e^{-sS}]$. We have
\begin{align}\label{ln4}
    F'(s)=\left(1-\frac{\lambda}{\mu} \right)\frac{\left(\tilde{S}(s)+s\tilde{S}'(s)\right)\left(s-\lambda+\lambda\tilde{S}(s)\right)-\left(s\tilde{S}(s)\right)\left(1+\lambda \tilde{S}'(s)\right)}{\left(s-\lambda+\lambda\tilde{S}(s)\right)^2},
\end{align}
and hence
\begin{align*}
     F'(\lambda)=\left(1-\frac{\lambda}{\mu} \right)\frac{\tilde{S}(\lambda)-1}{\lambda\tilde{S}(\lambda)}=\left(1-\frac{\lambda}{\mu} \right)\frac{\mathbb E[e^{-\lambda S}]-1}{\lambda\mathbb E[e^{-\lambda S}]}.
\end{align*}
And from \eqref{ln1}
\begin{align*}
     \mathbb E[D_ke^{-\lambda D_k}]=\left(1-\frac{\lambda}{\mu} \right)\frac{1-\mathbb E[e^{-\lambda S}]}{\lambda\mathbb E[e^{-\lambda S}]}.
\end{align*}
Therefore, from \eqref{ln2}
\begin{align}\label{ln8}
    \mathbb E[D_{k}|U_k>0]=\frac{1-\mathbb E[e^{-\lambda S}]}{\lambda \mathbb E[e^{-\lambda S}]},
\end{align}
and 
\begin{align}\label{lnD2}
   F_2= \mathbb E[|(S_{k+1}-X_{k+1})|U_k^2|U_k>0] \leq \frac{2}{\mu \lambda^2}+\frac{2}{\lambda^2}\frac{1-\mathbb E[e^{-\lambda S}]}{\lambda\mathbb E[e^{-\lambda S}]}+\frac{6}{\lambda^3}.
\end{align}
Next for $F_1$ we have again on $\{U_k >0\}$, $X_{k+1}=D_k+U_k$. That is
\begin{align*}
(S_{k+1}-X_{k+1})^2 = (S_{k+1}-D_k-U_k)^2 \leq 3(S_{k+1}^2+D_k^2+U_k^2).
\end{align*}
Therefore
\begin{align}\label{ln13}
F_1=\mathbb E[(S_{k+1}-X_{k+1})^2U_k|U_k>0] \leq 3\mathbb E[S_{k+1}^2U_k|U_k>0] + 3\mathbb E[D_{k}^2U_k|U_k>0]+ 3\mathbb E[U_k^3|U_k>0].
\end{align}
As before
\begin{align}\label{ln14}
\mathbb E[U_k^3|U_k>0] = \frac{6}{\lambda^3},
\end{align}
and since $S_{k+1}$ is independent of $U_k$,
\begin{align}\label{ln15}
\mathbb E[S_{k+1}^2U_k|U_k>0]=\mathbb E[S_{k+1}^2]\mathbb E[U_k|U_k>0]=\frac{2}{\mu^2\lambda}.
\end{align}
% Using the independence of $U_k$ and $D_k$ (conditioned on $U_k>0$), we have
% \begin{align}\label{ln6}
% \mathbb E[D_{k}^2U_k|U_k>0]=\mathbb E[U_k|U_k>0]\mathbb E[D_{k}^2|U_k>0]=\frac{1}{\lambda}\mathbb E[D_{k}^2|U_k>0]
% \end{align}
We have using the strong memoryless property 
\begin{align}\label{ln6}
     \mathbb E[D_{k}^2U_k|U_k>0]=\mathbb E[ \mathbb E[D_{k}^2U_k|D_k,U_k>0]|U_k>0]=\mathbb E[ D_{k}^2\mathbb E[U_k|D_k,U_k>0]|U_k>0]=\frac{1}{\lambda}\mathbb E[D_{k}^2|U_k>0].
\end{align}
Further, just as in \eqref{ln2}
\begin{align}\label{ln3}
\mathbb E[D_{k}^2|U_k>0] &= \frac{\mathbb E[D_k^2 e^{-\lambda D_k}]}{\mathbb E[e^{-\lambda D_k}]}.
\end{align}
Recall the PK formula $F(s) = \mathbb E[e^{- s D_k}]$ and $ \mathbb E[e^{-\lambda D_k}]=\left(1-\frac{\lambda}{\mu}\right)$. Differentiating $F(s)$ twice with respect to $s$ we get
\begin{align*}
F''(s) = \mathbb E[D_k^2 e^{-s D_k}].
\end{align*}
Thus, the numerator in \eqref{ln3} is $F''(\lambda)$.
From \eqref{ln4}
\begin{align*}
    F''(\lambda) = \left(1-\frac{\lambda}{\mu} \right)\frac{2(1 - \tilde{S}(\lambda) + \lambda \tilde{S}'(\lambda))}{(\lambda \tilde{S}(\lambda))^2},
\end{align*}
and therefore \eqref{ln3} gives
\begin{align}\label{ln5}
\mathbb E[D_{k}^2|U_k>0] =  \frac{2(1 - \tilde{S}(\lambda) + \lambda \tilde{S}'(\lambda))}{(\lambda \tilde{S}(\lambda))^2}=\frac{2\left(1 - \mathbb E[e^{-\lambda S}] - \lambda \mathbb E[S e^{-\lambda S}]\right)}{\lambda^2 (\mathbb E[e^{-\lambda S}])^2}.
\end{align}
Combining \eqref{ln13},\eqref{ln14}, \eqref{ln15}, \eqref{ln6}, \eqref{ln5}, we obtain 
\begin{align}\label{lnD1}
F_1=\mathbb E[(S_{k+1}-X_{k+1})^2U_k|U_k>0] \leq \frac{6}{\mu^2\lambda} + \frac{6\left(1 - \mathbb E[e^{-\lambda S}] - \lambda \mathbb E[S e^{-\lambda S}]\right)}{\lambda^3 (\mathbb E[e^{-\lambda S}])^2} + \frac{18}{\lambda^3}.
\end{align}
From \eqref{lnD}, \eqref{lnD1} and \eqref{lnD2}
\begin{align}\label{lnDbound}
    F \leq& \delta^4 \frac{1}{6}\| f'''\|\bigg[\frac{6}{\lambda^2}+\frac{18}{\mu^2} + \frac{18\left(1 - \mathbb E[e^{-\lambda S}] - \lambda \mathbb E[S e^{-\lambda S}]\right)}{\lambda^2 (\mathbb E[e^{-\lambda S}])^2} + \frac{54}{\lambda^2}+ \frac{6}{\mu \lambda}+\frac{6(1-\mathbb E[e^{-\lambda S}])}{\lambda^2\mathbb E[e^{-\lambda S}]}+\frac{18}{\lambda^2}\bigg] \nonumber\\
    =& \delta^4 \frac{1}{6}\| f'''\|\bigg[\frac{78}{\lambda^2}+\frac{18}{\mu^2} + \frac{18\left(1 - \mathbb E[e^{-\lambda S}] - \lambda \mathbb E[S e^{-\lambda S}]\right)}{\lambda^2 (\mathbb E[e^{-\lambda S}])^2} + \frac{6}{\mu \lambda}+\frac{6(1-\mathbb E[e^{-\lambda S}])}{\lambda^2\mathbb E[e^{-\lambda S}]}\bigg].
\end{align}
Next
 \begin{align}\label{lnB}
B &= \bigg| \delta^3 \mathbb E \bigg[\frac{1}{2}D_k^2U_kf'''(\eta_1)+\delta^2 D_k f'''(\eta_2) +\frac{1}{2}D_kU_k^2f'''(\eta_3)+D_kU_k(S_{k+1}-X_{k+1})f'''(\eta_4)\bigg] \bigg| \nonumber\\
       & \leq \delta^3 \|f'''\|\bigg[\frac{1}{2}\mathbb E[D_k^2U_k]+\delta^2\mathbb E[D_k]+ \frac{1}{2}\mathbb E[D_kU_k^2] +\mathbb E[D_kU_k|S_{k+1}-X_{k+1}|] \bigg] \nonumber\\
       & = \delta^3 \|f'''\|\bigg[\frac{1}{2}\mathbb E[D_k^2U_k|U_k>0]P(U_k>0)+\delta^2\mathbb E[D_k]+ \frac{1}{2}\mathbb E[D_kU_k^2|U_k>0]P(U_k>0) \nonumber\\
       &+\mathbb E[D_kU_k|S_{k+1}-X_{k+1}||U_k>0]P(U_k>0)\bigg] \nonumber\\
       & = \delta^4 \|f'''\|\bigg[\frac{1}{2}\lambda \mathbb E[D_k^2U_k|U_k>0]+\delta\mathbb E[D_k]+ \frac{\lambda}{2}\mathbb E[D_kU_k^2|U_k>0] +\lambda\mathbb E[D_kU_k|S_{k+1}-X_{k+1}||U_k>0] \bigg].      
\end{align}
From \eqref{ln6}, \eqref{ln5} we have
\begin{align}\label{ln7}
    E[D_k^2U_k|U_k>0]=\frac{2\left(1 - \mathbb E[e^{-\lambda S}] - \lambda \mathbb E[S e^{-\lambda S}]\right)}{\lambda^3 (\mathbb E[e^{-\lambda S}])^2}.
\end{align}
From \eqref{lindeq7} we have 
\begin{align}\label{lnB2}
    \delta\mathbb E[D_k]=\frac{1}{\lambda \mu}.
\end{align}
% Again, since conditional on $U_k > 0$, the random variable $U_k = X_{k+1} - D_k$ is exponentially distributed with rate $\lambda$ and is independent of $D_k$
% \begin{align}\label{ln9}
%     \mathbb E[D_kU_k^2|U_k>0]=\mathbb E[D_k|U_k>0]\mathbb E[U_k^2|U_k>0]=\frac{2}{\lambda^2}\mathbb E[D_k|U_k>0]=\frac{2(1-\mathbb E[e^{-\lambda S}])}{\lambda^3 \mathbb E[e^{-\lambda S}]}
% \end{align}
Again, using the strong memoryless property,
\begin{align}\label{ln9}
     \mathbb E[D_kU_k^2|U_k>0]=\mathbb E[\mathbb E[D_kU_k^2|D_k,U_k>0]|U_k>0]=&\mathbb E[D_k\mathbb E[U_k^2|D_k,U_k>0]|U_k>0] \nonumber\\
     =&\frac{2}{\lambda^2}\mathbb E[D_k|U_k>0]=\frac{2(1-\mathbb E[e^{-\lambda S}])}{\lambda^3 \mathbb E[e^{-\lambda S}]},
\end{align}
where the last equality follows from \eqref{ln8}. On the event $\{U_k > 0\}$, we again substitute $X_{k+1} = D_k + U_k$. Hence, $|S_{k+1} - X_{k+1}| = |S_{k+1} - D_k - U_k| \leq S_{k+1} + D_k + U_k$. As a result,

% \begin{align*}
%     \mathbb E[D_kU_k|S_{k+1}-X_{k+1}||U_k>0] &\leq  \mathbb E[D_kU_kS_{k+1}|U_k>0] +\mathbb E[D_k^2U_k|U_k>0] +\mathbb E[D_kU_k^2|U_k>0] \\
%     &\stackrel{(a)}{=}\mathbb E[S_{k+1}]\mathbb E[D_kU_k|U_k>0] +\mathbb E[D_k^2U_k|U_k>0] +\mathbb E[D_kU_k^2|U_k>0]\\
%     &\stackrel{(b)}{=}\mathbb E[S_{k+1}]\mathbb E[D_k|U_k>0]\mathbb E[U_k|U_k>0]  +\mathbb E[D_k^2U_k|U_k>0] +\mathbb E[D_kU_k^2|U_k>0]
% \end{align*}
\begin{align*}
    \mathbb E[D_kU_k|S_{k+1}-X_{k+1}||U_k>0] &\leq  \mathbb E[D_kU_kS_{k+1}|U_k>0] +\mathbb E[D_k^2U_k|U_k>0] +\mathbb E[D_kU_k^2|U_k>0] \\
    &\stackrel{(a)}{=}\mathbb E[S_{k+1}]\mathbb E[D_kU_k|U_k>0] +\mathbb E[D_k^2U_k|U_k>0] +\mathbb E[D_kU_k^2|U_k>0]\\
    &\stackrel{(b)}{=}\mathbb E[S_{k+1}]\mathbb E[D_k\mathbb E[U_k|D_k,U_k>0]|U_k>0]+\mathbb E[D_k^2U_k|U_k>0] \\
    &+\mathbb E[D_kU_k^2|U_k>0]\\
    &=\mathbb E[S_{k+1}]\frac{\mathbb E[D_k|U_k>0]}{\lambda}+\mathbb E[D_k^2U_k|U_k>0] +\mathbb E[D_kU_k^2|U_k>0],
\end{align*}
where in $(a)$ we have used the independence of $S_{k+1}$ with $D_k,U_k$, and in $(b)$ we have again used strong memoryless property of the interarrival times. From  \eqref{ln8}, \eqref{ln7} and \eqref{ln9},
\begin{align}\label{lnB4}
    \mathbb E[D_kU_k|S_{k+1}-X_{k+1}||U_k>0] &\leq \frac{1-\mathbb E[e^{-\lambda S}]}{\lambda^2\mu \mathbb E[e^{-\lambda S}]}+\frac{2\left(1 - \mathbb E[e^{-\lambda S}] - \lambda \mathbb E[S e^{-\lambda S}]\right)}{\lambda^3 (\mathbb E[e^{-\lambda S}])^2}+\frac{2(1-\mathbb E[e^{-\lambda S}])}{\lambda^3 \mathbb E[e^{-\lambda S}]}.
\end{align}
From \eqref{lnB},\eqref{ln7}, \eqref{lnB2}, \eqref{ln9} and \eqref{lnB4}, 
\begin{align}\label{lnBbound}
    B \leq& \delta^4 \|f'''\|\bigg[\frac{\left(1 - \mathbb E[e^{-\lambda S}] - \lambda \mathbb E[S e^{-\lambda S}]\right)}{\lambda^2 (\mathbb E[e^{-\lambda S}])^2}+\frac{1}{\lambda \mu} +\frac{(1-\mathbb E[e^{-\lambda S}])}{\lambda^2\mathbb E[e^{-\lambda S}]} \nonumber\\
    &+\frac{1-\mathbb E[e^{-\lambda S}]}{\lambda\mu \mathbb E[e^{-\lambda S}]}+\frac{2\left(1 - \mathbb E[e^{-\lambda S}] - \lambda \mathbb E[S e^{-\lambda S}]\right)}{\lambda^2 (\mathbb E[e^{-\lambda S}])^2}+\frac{2(1-\mathbb E[e^{-\lambda S}])}{\lambda^2 \mathbb E[e^{-\lambda S}]}\bigg] \nonumber\\
    =& \delta^4 \|f'''\|\bigg[\frac{3\left(1 - \mathbb E[e^{-\lambda S}] - \lambda \mathbb E[S e^{-\lambda S}]\right)}{\lambda^2 (\mathbb E[e^{-\lambda S}])^2}+\frac{1}{\lambda \mu} +\frac{3(1-\mathbb E[e^{-\lambda S}])}{\lambda^2\mathbb E[e^{-\lambda S}]}+\frac{1-\mathbb E[e^{-\lambda S}]}{\lambda\mu \mathbb E[e^{-\lambda S}]}\bigg].
\end{align}
And finally,
\begin{align}\label{lnEbound}
   E= \bigg|\frac{\delta^4}{4!}\mathbb E[f^{(4)}(\delta \Tilde{Z})(S_{k+1}-X_{k+1}+U_k)^4] \bigg| &\leq \frac{\delta^4}{4!} \|f^{(4)}\|\mathbb E[(S_{k+1}-X_{k+1}+U_k)^4] \nonumber\\
   & \leq \frac{\delta^4}{4!} \|f^{(4)}\| \big[ 8 (\mathbb E[S_{k+1}^4]+16\mathbb E[X_{k+1}^4])\big] \nonumber\\
   &=\frac{8\delta^4}{4!} \|f^{(4)}\| \bigg(\mathbb E[S^4] +16\frac{4!}{\lambda^4}\bigg).
\end{align}
From \eqref{ln12}, \eqref{lnBbound}, \eqref{lnCbound}, \eqref{lnDbound}, \eqref{lnEbound},
\begin{align*}
     &\delta^2\bigg|- \mathbb E[f'(\delta D_k)]+\frac{1}{\lambda \mu} \mathbb E[f''(\delta D_k)]+f'(0)\bigg| \nonumber\\
    &\leq \delta^4 \|f'''\|\bigg[\frac{3\left(1 - \mathbb E[e^{-\lambda S}] - \lambda \mathbb E[S e^{-\lambda S}]\right)}{\lambda^2 (\mathbb E[e^{-\lambda S}])^2}+\frac{1}{\lambda \mu} +\frac{3(1-\mathbb E[e^{-\lambda S}])}{\lambda^2\mathbb E[e^{-\lambda S}]}+\frac{1-\mathbb E[e^{-\lambda S}]}{\lambda\mu \mathbb E[e^{-\lambda S}]}\bigg]+\delta^4 \|f'''\|\left(\frac{1}{\mu^2}+\frac{1}{\lambda^2}\right)\\  
    &+ \delta^4 \frac{1}{6}\| f'''\|\bigg[\frac{78}{\lambda^2}+\frac{18}{\mu^2} + \frac{18\left(1 - \mathbb E[e^{-\lambda S}] - \lambda \mathbb E[S e^{-\lambda S}]\right)}{\lambda^2 (\mathbb E[e^{-\lambda S}])^2} + \frac{6}{\mu \lambda}+\frac{6(1-\mathbb E[e^{-\lambda S}])}{\lambda^2\mathbb E[e^{-\lambda S}]}\bigg]+\frac{8\delta^4}{4!} \|f^{(4)}\| \bigg(\mathbb E[S^4] +16\frac{4!}{\lambda^4}\bigg).
\end{align*}
Therefore,
\begin{align*}
    &\bigg|- \mathbb E[f'(\delta D_k)]+\frac{1}{\lambda \mu} \mathbb E[f''(\delta D_k)]+f'(0)\bigg| \nonumber\\
    &\leq \delta^2 \|f'''\| \bigg[\frac{6\left(1 - \mathbb E[e^{-\lambda S}] - \lambda \mathbb E[S e^{-\lambda S}]\right)}{\lambda^2 (\mathbb E[e^{-\lambda S}])^2}+\frac{2}{\lambda \mu} +\frac{4(1-\mathbb E[e^{-\lambda S}])}{\lambda^2\mathbb E[e^{-\lambda S}]}+\frac{1-\mathbb E[e^{-\lambda S}]}{\lambda\mu \mathbb E[e^{-\lambda S}]}+\frac{14}{\lambda^2}+\frac{4}{\mu^2} \bigg]\\
    &+\frac{8\delta^2}{4!} \|f^{(4)}\| \bigg(\mathbb E[S^4] +16\frac{4!}{\lambda^4}\bigg).
\end{align*}
\end{proof}
\section{Proof of Lemma \ref{lemmadistancebound}}\label{app:c}
\begin{proof}[Proof of Lemma \ref{lemmadistancebound}]
    % We have 
    % \begin{align}\label{bound Wasserstein}
    %     d_{W,k}(Z, D) \leq 2(2^{(k-2)}k.d_{Zol,k}(Z, D))^{\frac{1}{k}}.
    % \end{align}
From \cite[Remark 1]{HB} (with $l=k-1$ and $\beta=1$) we have 
\begin{align}\label{bound rho_k}
     d_{W,k}(\mu, \sigma) \leq 2(2^{(k-2)}k\rho_k(\mu, \sigma))^{\frac{1}{k}},
\end{align}
where
\begin{align*}
   \rho_k (\mu,\sigma)=\sup_{\substack{\|h^{(k)}\| \leq 1 \\ h^{(i)}(0)=0,\; 0 \leq i \leq k-1}} \bigg|\int h d\mu-\int h d\sigma\bigg|.
\end{align*}
If the first $k-1$ moments of $\mu$ and $\sigma$ do not match we have $d_{Zol,k}(\mu, \sigma))=\infty$ (see Remark \ref{remark1}) and \eqref{bound Wasserstein} holds. And if the first $k-1$ moments match we can take 
\begin{align*}
    \Tilde{h}(x):=h(x)-\sum_{i=0}^{k-1}\frac{h^{(i)}(0)}{i!}x^i.
\end{align*}
Hence, $\Tilde{h}^{(i)}(0)=0$ for $0 \leq i \leq k-1$ and $\Tilde{h}^{(k)}=h^{(k)}$. Further,
\begin{align*}
   \bigg|\int \tilde{h} d\mu-\int \tilde {h} d\sigma\bigg|&=\bigg|\int h d\mu-\int h d\sigma-\sum_{i=0}^{k-1}\frac{h^{(i)}(0)}{i!}\bigg(\int x^id \mu-\int x^i d\sigma\bigg) \bigg|\\
    &=\bigg|\int h d\mu-\int h d\sigma\bigg|.
\end{align*}
Therefore, if the first $k-1$ moments match we have 
\begin{align}\label{distanceequal}
    d_{Zol,k}(\mu, \sigma)=\rho_k (\mu,\sigma)
\end{align}
and \eqref{bound Wasserstein} follows from \eqref{bound rho_k}.

In the other direction if the first $k-1$ moments match, \cite{BDS} and \eqref{distanceequal} gives that
\begin{align}
    d_{Zol,k}(\mu,\sigma) \leq L_k d_{W,k}(\mu, \sigma)\bigg(\int |x|^k \mu(dx) +\int |x|^k \sigma(dx)\bigg)^{(k-1)/k}.
\end{align}
\end{proof}

\end{document}